\documentclass[10pt]{article}

\usepackage{amsmath}
\usepackage{amssymb}

\usepackage{graphicx}

\usepackage[colorlinks=true]{hyperref}

\usepackage{color} 

\usepackage[labelfont=bf,labelsep=period,justification=raggedright]{caption}

\makeatletter
\renewcommand{\@biblabel}[1]{\quad#1.}
\makeatother

\date{}

\graphicspath{{Figures/}}
\newtheorem{theorem}{Theorem}[section]
\newtheorem{lemma}[theorem]{Lemma}

\newtheorem{remark}{\it Remark\/}

\newcommand{\mF}{\mathcal{F}}

\newcommand{\R}{\mathbb{R}}

\newcommand{\suml}{\sum\limits}

\newcommand{\norm}[1]{\left \| #1 \right \|}

\usepackage{ifthen}
\newboolean{DisplaySOS}
\setboolean{DisplaySOS}{false} 
\newcommand{\SOS}[1]{\ifthenelse{\boolean{DisplaySOS}}{{\bf[#1]}}{}}

\begin{document}

\begin{flushleft}
{\Large
\textbf{
Illusions in the Ring Model of visual orientation selectivity
}
}
\\
Romain Veltz$^{1,2}$, 
Olivier Faugeras$^{2}$, 
\\
\bf{1} IMAGINE/LIGM, Universit\'e Paris Est., France
\\
\bf{2} NeuroMathComp team, INRIA, CNRS, ENS Paris, France
\\
$\ast$ E-mail: romain.veltz@sophia.inria.fr
\end{flushleft}

\section*{Abstract}

\SOS{Please keep the Author Summary between 150 and 200 words
 Use first person. PLoS ONE authors please skip this step. 
 Author Summary not valid for PLoS ONE submissions.}  
 
 \SOS{Il y a 270 mots...}
 The Ring Model of orientation tuning is a dynamical model of a hypercolumn of visual area V1 in the human neocortex that has been designed to account for the experimentally observed orientation tuning curves by local, i.e., cortico-cortical computations. The tuning curves are stationary, i.e. time independent, solutions of this dynamical model. One important assumption underlying the Ring Model is that the LGN input to V1 is weakly tuned to the retinal orientation and that it is the local computations in V1 that sharpen this tuning. Because the equations that describe the Ring Model have built-in equivariance properties in the synaptic weight distribution with respect to a particular group acting on the retinal orientation of the stimulus, the model in effect encodes an infinite number of tuning curves that are arbitrarily translated with respect to each other. By using the Orbit Space Reduction technique we rewrite the model equations in canonical form as functions of polynomials that are invariant with respect to the action of this group. This allows us to combine equivariant bifurcation theory with an efficient numerical continuation method in order to compute the tuning curves predicted by the Ring Model. Surprisingly some of these tuning curves are not tuned to the stimulus. We interpret them as neural illusions and show numerically how they can be induced by simple dynamical stimuli. These neural illusions are important biological predictions of the model. If they could be observed experimentally this would be a strong point in favor ot the Ring Model. We also show how our theoretical analysis allows to very simply specify the ranges of the model parameters by comparing the model predictions with published experimental observations.
\section*{Author Summary}

\SOS{TO DO}
The visual cortex of humans features a columnar organization. Each column contains cells whose receptive fields monitor almost identical retinal positions and orientations. Such nearby {\it orientation columns} encode different orientations. A set of orientation columns encoding all possible orientations is a hypercolumn of orientation. The visual input to a hypercolumn is from the retina through the thalamus. The hypercolumn is a relay that sharpens the output of the thalamus to make the local retinal visual orientation more "conspicuous". This happens because the neurons of the hypercolumn inhibit and excite each other according to simple rules. These rules are used to build a mathematical model of a hypercolumn whose predictions can be compared with experimental evidence. We study the properties of such a model whose rich symmetries can be traced to the fact that two orientations differing by the value $\pi$ are the same visually. Important consequences of our analysis are the clarification of the model parameters role in shaping the ``perception'' of the hypercolumn and the prediction of a neuronal illusion, the fact that a hypercolumn can be in the same state as if the retinal stimulus were at an orientation rotated by 90 degrees with respect to the actual one.

\section*{Introduction}\label{section:intro}
Since the discovery by Hubel and Wiesel \cite{hubel-wiesel:62} of the selective response of a single neuron to some orientations, a long standing debate has been the degree of cortical computation involved in this selectivity compared to the feedforward selectivity implied by the LGN projections. Cortical models  \cite{somers-nelson-etal:95,ben-yishai-bar-or-etal:95,hansel-sompolinsky:97} have been used to show how this selectivity can be produced in a cortex with center-surround interactions in the orientation domain.

The  Ring Model of orientation tuning was introduced by Hansel and Sompolinski \cite{hansel-sompolinsky:97} and studied by several other scientists \cite{shriki-hansel-etal:03,ermentrout:98,dayan-abbott:01,bressloff-bressloff-etal:00,bressloff-cowan-etal:01}, after the seminal work of Ben-Yishai and colleagues \cite{ben-yishai-bar-or-etal:95}, as a model of a hypercolumn in primary visual cortex. This rate model is a simplification of complex spiking networks \cite{somers-nelson-etal:95,douglas-koch-etal:95} designed to make it easier to understand the role of different mesoscopic parameters.

It assumes that the local orientation $x$ in the receptor fields of the neurons in the column is encoded in their activity, or firing rate, noted $A(x,t)$. The interaction between the neurons is modeled by a function $J$ of the orientation that represents how the activities corresponding to two different orientations reinforce or inhibit each other. This function is called the connectivity function of the model. With this in mind, the dynamics of the firing rate can be represented by the following integro-differential\SOS{ce n'est pas une equation integro-differentielle, il n y a pas de terme differentiel, c'est juste une ODE} equation in the line of the model of Wilson-Cowan \cite{wilson-cowan:72}:
\begin{equation} \label{eq:ringa}
\left\{
\begin{array}{lcl}
\tau\dot A(x,t)&=& -A(x,t) + 
S\left[ \lambda \Big(\int\limits_{-\pi/2}^{\pi/2}  J(x-y)
A(y,t))dy/\pi+\varepsilon I(x)-\theta\Big)\right] \quad t>0\\
A(x,0)&=&A_0(x)
\end{array}
\right.
\end{equation}
$\tau$ defines the intrinsic dynamics of the population and $I$ is the input from the lateral geniculate nucleus (LGN) to the hypercolumn whose contrast is defined by the parameter $\varepsilon$, see figure \ref{fig:visual path}. $S$ is the sigmoid
\[
 S(x)=\frac{1}{1+e^{- x}},
\]
which takes values between 0 and 1, $\lambda$ is a parameter that determines the nonlinear gain of the sigmoid, $\theta$ is a threshold that controls for which value the sigmoid takes the value $1/2$, the function $I$ represents the input from the LGN.
\begin{figure}
 \centerline{
\includegraphics[width=0.75\textwidth]{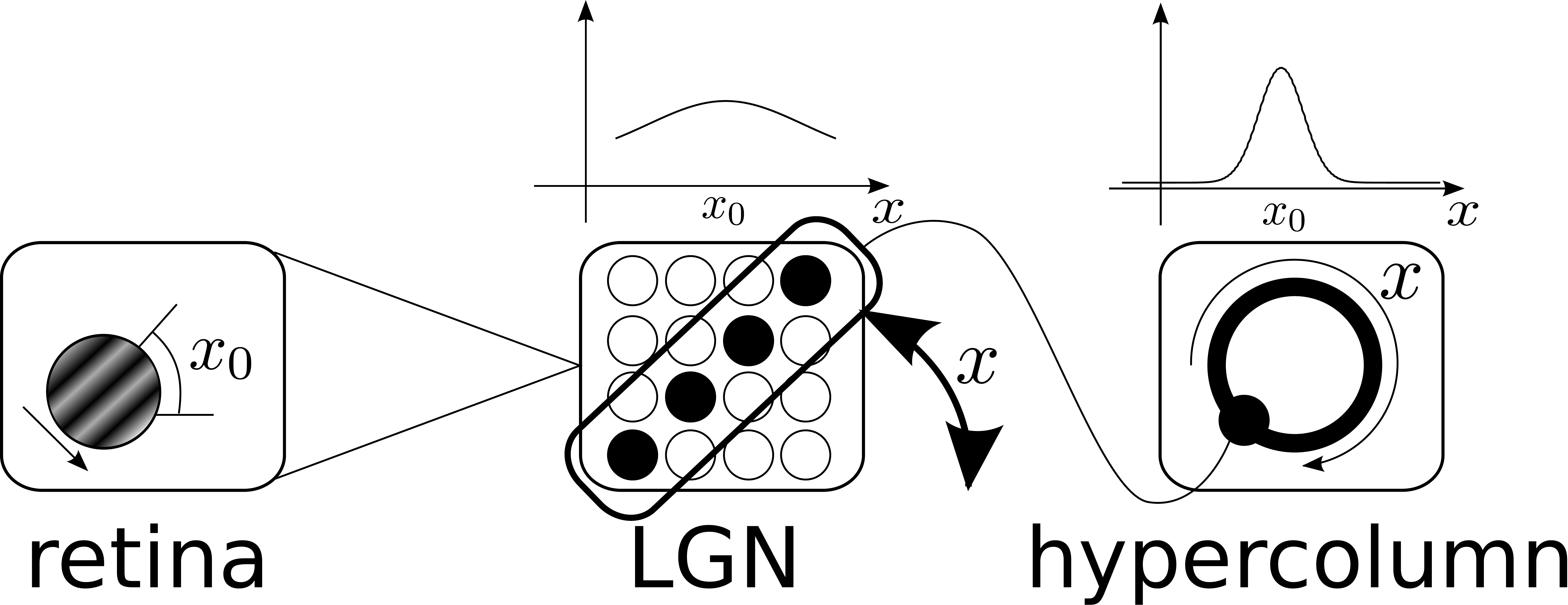}
}
\caption{A simplified view of the visual path from the retina through the LGN to cortical area V1. The receptive field of the LGN cells feeding the hypercolumn of orientation of V1 contains a grating of orientation $x_0$. This grating excites mostly the LGN cells that share this receptive field and are aligned in the direction $x_0$; the tuning is broad, see the curve in the middle part of the figure. These LGN cells project onto the network of cells in the hypercolumn of orientation of V1 whose interactions, represented by the Ring Model, result in a sharpening of the tuning around the grating direction, see the curve in the righthand side of the figure.}
\label{fig:visual path}
\end{figure}

Regarding the connectivity function, in all cases it is an even $\pi$-periodic function, positive for orientation values close to 0 (corresponding to an excitation) and negative for orientation values larger in magnitude (corresponding to an inhibition). The rich symmetries of $J$ play a prominent role in our upcoming analysis of the Ring Model. The reason for this is that, when the contrast $\varepsilon$ is equal to 0, equation \eqref{eq:ringa}, is equivariant, i.e. it has some nice properties with respect to the action of a certain group that we proceed to describe. 

Let us consider the group, noted $T$, of translations of $\R/2\pi \mathbb{Z}$. An element $T_\gamma$, $\gamma \in \R/2\pi \mathbb{Z}$ of this group acts on the orientation $x$ by $T_\gamma \cdot x=\gamma/2+x$ and on the activity function $A(x,t)$ by $T_\gamma \cdot A(x,t)=A(\gamma/2+x,t)$. Similarly, we consider the reflection, noted $R$, such that $R \cdot x=-x$ and  therefore $R \cdot A(x,t)=A(-x,t)$. We note $G$ the group generated by $T$ and $R$. We can abstractly rewrite equation \eqref{eq:ringa} as $\mF(x)=0$, where $\mF(x)=\tau\dot A(x,t)+A(x,t) - S\left[ \lambda \Big(\int\limits_{-\pi/2}^{\pi/2}  J(x-y)A(y,t))dy/\pi+\varepsilon I(x)-\theta\Big)\right]$. Using the symmetry properties of the function $J$ it is easy to verify that $\mF$ satisfies the condition $\mF(g \cdot x)=g \cdot \mF(x)$ for all elements $g$ in the group $G$ when $\varepsilon=0$. Since moreover we clearly have
\[
\left\{
 \begin{array}{lclclcl}
  T_{\gamma_1} T_{\gamma_2} & = & T_{\gamma_1+\gamma_2} & RT_\gamma & = & T_{-\gamma} R \quad \text{for all} \quad \gamma_1,\,\gamma_2,\,\gamma \in \R/2\pi \mathbb{Z}\\
T_0 &=& Id & R^2 &= & Id
 \end{array}
\right.
\]
the group G is in effect isomorphic to $O(2)$\SOS{l isomorphisme est evident non ? il suffit d associer a $T_\gamma$ la matrice de rotation 2x2 d'angle $\gamma$ et a R, la reflexion 2x2}, the group of two-dimensional othogonal transformations \cite[Chapter 1]{haragus-iooss:09}.

Several variants of this model have been studied in the literature, e.g., in \cite{bressloff-bressloff-etal:00,bressloff-cowan-etal:01} $J$ is a difference of Gaussians while in \cite{ben-yishai-bar-or-etal:95} the authors start with a network of excitatory/inhibitory spiking neurons and derive a meanfield approximation of this network yielding an interaction function $J$ described by the following equation:
\begin{equation}\label{eq:J}
   J(x)=  J_0+J_1 \cos(2x)\\
\end{equation}

\noindent
The input $I$ from the LGN has a similar shape
\begin{equation}\label{eq:I}
 I(x) =  1-\beta+\beta \cos(2(x-x_0)).
\end{equation}
As mentioned above, and as shown in figure \ref{fig:visual path}, it is weakly tuned, i.e., maximal,  at $x=x_0$ and it is the network, modeled by \eqref{eq:ringa},  that sharpens this tuning.
These authors vary the anisotropy parameter $\beta$ between 0 and 1.

The sigmoid $S$ is  often chosen to be a Heaviside function, or, as in \cite{ben-yishai-bar-or-etal:95}, a piecewise linear approximation of the sigmoid. In \cite{ermentrout:98,dayan-abbott:01,bressloff-bressloff-etal:00} it is a true sigmoidal function.
 
The parameter $J_1$ is positive, an important property of the network that is necessary to produce the tuning curves.\SOS{il y avait : as the functional connectivity featuring local excitation/lateral inhibition is preserved}. $J_0$ is most of the time negative  \cite{ben-yishai-bar-or-etal:95,dayan-abbott:01,bressloff-bressloff-etal:00,bressloff-cowan-etal:01} but can be positive as well  \cite{bressloff-cowan-etal:01}. Note that the $J_i$s, $i=0,\,1$ are the first Fourier coefficients of the function $J$. $J_0$ is its mean value and can be positive even if the surround is inhibitory.
For example, in \cite{dayan-abbott:01}, we find $J_0=-7.3,\ J_1=11,\ \beta=0.1,\ \theta=0$ which are the values in \cite{ben-yishai-bar-or-etal:95} except for $\theta=1$ and $\varepsilon=0.164$. The slope, or nonlinear gain, $\lambda$ is assumed to be equal to 1. Using the previous rescaling, it becomes $J_0=-1,\ J_1=1.5,\ \lambda=7.3/S'(0)=29.2$ and $\theta\to\theta/7.3 \approx 0.1$ in the case of \cite{ben-yishai-bar-or-etal:95} and $\theta=0$ in the case of \cite{dayan-abbott:01}.

The model \eqref{eq:ringa} is called an activity model in the terminology of \cite{ermentrout:98}. For technical reasons we turn it into a voltage model as follows.
We first rewrite  equation \eqref{eq:ringa} in a more compact and convenient, functional, form :
\[
 \tau\dot A=-A+S\left[ \lambda(J \cdot A+\varepsilon I-\theta)\right] .
\]
$J$ is now thought of as a linear (integral) operator acting on the function $A$ as the periodic convolution $J \cdot A(x,t)=\int\limits_{-\pi/2}^{\pi/2}  J(x-y)
A(y,t))dy/\pi$, see, e.g., \cite{veltz-faugeras:09}.
We then perform the change of variable $V=J \cdot A+ \varepsilon I-\theta$. Assuming that the input current is not a function of time, this leads to the following equation
\begin{equation}\label{eq:ringb}
   \tau\dot V=-V+J \cdot S(\lambda V)+\varepsilon I-\theta
 \end{equation}
Note that this equation, as \eqref{eq:ringa}, is $G$-equivariant.

The stationary solutions (some of them called tuning curves, see the paragraph \textit{Discussion}) of \eqref{eq:ringa} (respectively of \eqref{eq:ringb}) satisfy $\dot{A}=0$ (respectively $\dot{V}=0$). Characterizing and computing them for different values of the parameters is the first step toward understanding the dynamics of the solutions to these equations. Indeed, it is known that this type of equations is such that there are only heteroclinic (linking two stationary solutions, or equilibria) or unbounded orbits. Since we showed in \cite{veltz-faugeras:09} that all trajectories were bounded, this implies that they are made of heteroclinic orbits. This motivates further the study of the stationary solutions of \eqref{eq:ringb}. One of our goals is to show how the stationary solutions are organized and to give indications about the dynamics in a given range of parameters, corresponding to biologically plausible values. This  is relevant because many large scale models of V1 including many hypercolumns represent them with the Ring Model. Therefore a good understanding of one hypercolumn paves the way to an understanding of a population thereof. We show that, depending on the nonlinear gain $\lambda$, there may  exist many stationary solutions, which are all acceptable responses of the network to a given input from the LGN (at least for the model at hand). Thus this local orientation tuning device may behave less trivially than what it was initially designed for. In effect, the existence of these stationary solutions, sometimes called persistent cortical states, can make the local dynamics quite intricate when $\lambda$ is large enough to support the existence of these extra solutions.

\noindent
It is worth noticing that the stationary solutions of \eqref{eq:ringa} and \eqref{eq:ringb} are in one to one correspondence. As a consequence we will work on \eqref{eq:ringb} because it is mathematically more convenient.

\noindent
We will follow a method similar to the one developed in \cite{veltz-faugeras:09} to compute the stationary states of \eqref{eq:ringb}.  The method has been modified to take into account the symmetries of the Ring Model.
The general idea is that the LGN input is weak and only modulates the network activity. Hence the cortical network (represented by the Ring Model) encodes the possible tuning curves within its connectivity function and when presented with a weak external input, produces small deviations of 'its' tuning curves. Our goal is to compute these tuning curves.
However, because of the symmetries of the connectivity function $J$, the model in effect encodes an infinite number of tuning curves, and this is an endless cause of numerical problems. Indeed we pointed out above that if the input current was null, equation \eqref{eq:ringb} (respectively \eqref{eq:ringa}) was $G$-equivariant. This implies that if $V(x)$ is a stationary solution of \eqref{eq:ringb}, so are $V(x+\gamma/2)$, $\gamma \in \R/2\pi \mathbb{Z}$ and $V(-x)$. We show that by performing an appropriate change of variables, we can get rid of this redundancy and recover numerical accuracy. 
\section*{Materials and Methods}
\subsection*{Turning the problem into a finite dimensional one}\label{section:finite}
Problem \eqref{eq:ringa} (respectively \eqref{eq:ringb}) is infinite dimensional since the solutions live in some (unspecified, but a priori infinite dimensional) functional space. By truncating the Fourier series of the even $\pi$-periodic connectivity function $J$ we reduce the problem to a finite number of dimensions. We write
\[
 J(x) = J_0+\suml_{p=1}^NJ_p\ \cos(2px)
\]
where $N$ is the number of Fourier modes that are sufficient to represent $J$. We will show how the choice of $N$ affects the biological functional properties of the Ring Model. Notice that by varying $N$, we generate a family of models that contains all previously published ones. 
\begin{remark}
 For convenience, we shall write $\cos_k$ for the function $x\to \cos(kx)$. The same holds for $\sin_k$.
\end{remark}
It was shown in \cite{veltz-faugeras:09} that this form of the connectivity function implies that the solutions to \eqref{eq:ringb} can be written $V(x,t)=V^\parallel(x,t)+V^\bot(x,t)$, where $V^\parallel$ is a linear combination of the functions $\cos_{2p}$ and $\sin_{2p}$, $p=0,\cdots,N$ and the function $V^\bot$ tends to 0 exponentially fast when $t \to \infty$. This implies that the stationary solutions satisfy $V^\bot=0$.
\begin{remark}
 The spectrum of the integral operator $J$ associated with the eponymic connectivity function is readily seen to be equal to $\Sigma( J)=\left\lbrace J_0,\frac{J_p}{2} \right\rbrace_{1\leq p\leq N} $.
\end{remark}

\noindent
We can, up to a rescaling of $\lambda$ in \eqref{eq:ringa}, assume that $J_0$ takes the values $\pm1$:
\[
J_0\overset{\rm def}{=}\varepsilon_0\in\left\lbrace-1,1 \right\rbrace.
\]
Similarly we define $\varepsilon_k=\pm 1$ by
\[
 J_k=\varepsilon_k |J_k|, \, \varepsilon_k \in\left\lbrace-1,1 \right\rbrace \quad k=1,\cdots,N
\]
With all this in hands, the tuning curves satisfy the equation:
\begin{equation}
\label{eq:ringbtc}
 V(x) =  \int\limits_{-\pi/2}^{\pi/2}  \left[J_0+\suml_{p=1}^NJ_p\ \cos_{2p}(x-y)\right]
S\left[ \lambda V(y,t))\right]dy/\pi+\varepsilon I(x)-\theta 
\end{equation}
It follows that any stationary solution to \eqref{eq:ringb} can be written
\[
 V(x)=v_0+\suml_{p=1}^N\sqrt{|J_p|}\left[v_p^{(1)}\ \cos_{2p}(x)+ v_p^{(2)}\ \sin_{2p}(x)\right],
\]
where  $v_0,v_p^{(1),\,(2)}$, $p=1,\cdots,N$ are $2N+1$ reals. Solving \eqref{eq:ringbtc} is therefore equivalent to finding these reals. In the case of a general solution, $V^\parallel(x,t)$ is given by the same formula where the coefficients are now real functions of time.
Under the assumption that $V^\bot$ is neglected, it is easy to obtain the system of ordinary differential equations that are satisfied by the functions $v_0,v_p^{(1),\,(2)}$. Using the complex values  $z_k\overset{\rm def}{\equiv} v_k^{(1)}+iv_k^{(2)}$, $k=1,\cdots,N$ we obtain the following equations

\begin{equation}\label{eq:discrete}
\left\lbrace 
\begin{array}{lcl}
 \dot v_0+v_0&=&\frac{\varepsilon_0}{\pi}\int\limits_{-\frac{\pi}{2}}^\frac{\pi}{2}S\left[ \lambda v_0+\lambda\suml_{p=1}^N\sqrt{|J_p|}\Re\left(  z_pe^{-2piy}  \right) \right] dy -\theta +\varepsilon I_0\\
&& \quad \overset{\rm def}{=} B_0(v_0,\left\lbrace z_p\right\rbrace ) -\theta+\varepsilon I_0 \\
\dot z_k+z_k&=& \varepsilon_k\frac{\sqrt{|J_k|}}{\pi}\int\limits_{-\frac{\pi}{2}}^\frac{\pi}{2}S\left[ \lambda v_0+\lambda\suml_{p=1}^N\sqrt{|J_p|}\Re\left(  z_pe^{-2piy}  \right) \right] e^{2kiy}dy+\varepsilon I_k\\
&& \quad \overset{\rm def}{=} B_k(v_0,\left\lbrace z_p\right\rbrace )+\varepsilon I_k \quad k=1,\cdots,N
\end{array}\right.
\end{equation}
where $I(x)\overset{\rm def}{=}I_0+\suml_{k=1}^NI_k\sqrt{|J_k|}e^{2ikx}$ and $I_0\in\mathbb R,\ I_k\in\mathbb C$.

The coefficients defining the tuning curves satisfy the following equations:
\begin{equation*}
 \left\lbrace 
\begin{array}{l}
 v_0= B_0(v_0,\left\lbrace z_p\right\rbrace ) -\theta +\varepsilon I_0\\
z_k= B_k(v_0,\left\lbrace z_p\right\rbrace )+\varepsilon I_k\quad k=1,\cdots,N\\
\end{array}\right. 
\end{equation*}
The $N+1$ dimensional vector $(v_0,z_1,\cdots,z_N)$ is a representation of $V^\parallel$.
The group $G$ also acts on this representation as follows
\[
 \begin{array}{lcl}
 	T_\gamma\cdot(v_0,z_1,z_2,\cdots,z_N)&=&(v_0,e^{2i\gamma}z_1,e^{4i\gamma}z_2,\cdots,e^{2iN\gamma}z_N)\quad \gamma \in \R\\
 	R\cdot(v_0,z_1,z_2,\cdots,z_N)&=&(v_0,\bar z_1,\bar z_2,\cdots,\bar z_N)
 \end{array}
\]
\SOS{Olivier: ça ne veut rien dire que de dire qu'un groupe qui est défini de manière intrinsèque dépend de $N$. J'ai reformulé.}
%
As shown in the introduction, the use of the group $O(2)$ is motivated by the fact that when $\varepsilon=0$, then \eqref{eq:discrete} are $O(2)$-equivariant\footnote{If we write \eqref{eq:discrete} when $\varepsilon=0$ as $\frac{d}{dt}V^\parallel=F(V^\parallel)$, this means that $\forall g\in O(2),\forall V^\parallel$, $F(g\cdot V^\parallel)=g\cdot F(V^\parallel)$}. Since if $V^f$ is a stationary solution of \eqref{eq:discrete} for $\varepsilon=0$, so is $g\cdot V^f,\ \forall g\in O(2)$ there is an infinity of tuning curves that are encoded by the network. However, when $\varepsilon\prod\limits_{k=1}^N I_k\neq 0$, all the symmetries are broken, \eqref{eq:discrete} are not $O(2)$-equivariant anymore and the number of tuning curves becomes finite.

In the next two sections we study the cases $N=1$ and $N=2$. The second case shows that adding more modes does not change the main results of the analysis.
\subsection*{Keeping only one mode in $J$}
We consider the following connectivity function
\[
 J=\varepsilon_0+J_1\cos_2,\,J_1>0
\]
From our previous analysis of the symmetries of the Ring Model, we know that the equations \eqref{eq:discrete} are redundant when $\varepsilon=0$. In order to eliminate this redundancy they should be rewritten using their equivariant structure with respect to the action of the group $O(2)$. In this case, this turns out to be equivalent to writing an equation for $v_0$ and the magnitude $\rho$ of $z_1$. It is convenient to write $z_1 = v_1+iv_2=\rho e^{2i\varphi}$, which yields to the following equations, assuming $x_0=0$:
\begin{equation}
\label{eq:onemode}
\left\{
    \begin{array}{lcl}
       \dot v_0 &=& -v_0+\varepsilon_0 B_0(v_0,\rho)-\theta+\varepsilon(1-\beta)\\
       \dot\rho &=& -\rho+J_1B_1(v_0,\rho)+ \frac{\varepsilon\beta}{\sqrt{J_1}}\cos_2(\varphi)\\
       2\rho\dot\varphi &=&-\sin_2(\varphi)\frac{\varepsilon\beta}{\sqrt{J_1}}
       \end{array}
\right.
\end{equation}
for the dynamics, and
\begin{equation}
\label{eq:onemodestatic}
\left\{
    \begin{array}{lcl}
       v_0 &=& \varepsilon_0 B_0(v_0,\rho)-\theta+\varepsilon(1-\beta)\\
      \rho &=& \sqrt{J_1}B_1(v_0,\rho)+ \frac{\varepsilon\beta}{\sqrt{J_1}}\cos_2(\varphi)\\
       0&=&\sin_2(\varphi)\frac{\varepsilon\beta}{\sqrt{J_1}}
       \end{array}
\right.
\end{equation}
for the tuning curves.
The functions $B_0$ and $B_1$ are given by (using an integration by parts to factor out $\rho$ in $B_1(v_0,\rho)$)
\[ 
    \begin{array}{lcl}
     B_0(v_0,\rho)&=&\frac{1}{\pi}\int\limits_{-\frac{\pi}{2}}^\frac{\pi}{2}S(\lambda(v_0+\sqrt{J_1}\rho\ \cos_2 x ))\,\frac{dx}{\pi}\\
     B_1(v_0,\rho)&=&\frac{1}{\pi}\sqrt{J_1}\lambda\rho\int\limits_{-\frac{\pi}{2}}^\frac{\pi}{2}S'(\lambda(v_0+\sqrt{J_1}\rho\ \cos_2 x ))\sin_2^2 x\,\frac{dx}{\pi}
    \end{array}
\]
Equations \eqref{eq:onemode} do not produce the same dynamics as \eqref{eq:discrete} because the change from Cartesian to polar cordinates is not a diffeomorphism. Nevertheless equations \eqref{eq:onemodestatic} are most useful for computing the tuning curves.

\subsection*{Case of 2 modes, $N=2$}

We now write:
\[
 J=\varepsilon_0+J_1\cos_2+J_2\cos_4
\]
 In order to agree with known experimental facts, the tuning curves should be mainly unimodal. Compared to the previous case, the fact that the second mode is nonzero could induce an ``interaction'' between the two modes leading to multimodal tuning curves.

Following the analysis of section \ref{section:finite} we have to solve five coupled equations which are redundant because of the action of the group $G$ which in this case reads
\[ \left\lbrace 
    \begin{array}{lcl}
       T_\gamma\cdot(v_0,z_1,z_2)&=&(v_0,e^{2i\gamma}z_1,e^{4i\gamma}z_2)\\
       R\cdot(v_0,z_1,z_2)&=&(v_0,\bar z_1,\bar z_2)
       \end{array}\right.
\]
In order to eliminate the redundancy arising from this symmetry we used polar coordinates as in the case $N=1$. It is tempting to do the same with the two complex variables $z_1$ and $z_2$ but it turns out to be a dead end.\\

The main reason is numerical: we compute (see \textit{Supporting information}, second paragraph) the solutions of the nonlinear equations \eqref{eq:discrete} using numerical continuation. This scheme works well if the Jacobian of the nonlinear equation has at worst a one-dimensional kernel at some isolated points. If we write $z_1=\rho_1e^{2i\varphi_1},\ z_2=\rho_2e^{4i\varphi_2}$, then, using the invariance by $T_\gamma$, the equations for $\dot\varphi_i=0$, $i=1,2$ simplify to $0=F_i(v_0,\rho_1,\rho_2,\varphi_1-\varphi_2)$ which are functions of the phase difference $\varphi_1-\varphi_2$. It turns out that the other equations (for $v_0,\rho_1,\rho_2$) also involve only $\varphi_1-\varphi_2$. We were unable to find a simple relation between $F_1$ and $F_2$. As a consequnce  we end up with 5 equations in the 4 unknows $v_0,\rho_1,\rho_2,\varphi_1-\varphi_2$: this is unappropriate for numerical continuation and we need to find a way to obtain 4 equations in 4 unknowns.

\

To reach this goal we turn to a general technique,  the Orbit Space Reduction \cite{chossat-lauterbach:00}, which provides the right change of coordinates through the use of what is called a Hilbert Basis\footnote{The ring $\mathcal R_G$ of $O(2)$-invariant polynomials is finitely generated as an $\mathbb R$-algebra, this goes back to Hilbert. A family $\pi_1,\cdots,\pi_s$ of generators of $\mathcal R_G$ is called a Hilbert Basis.}. A fundamental property is that any smooth equivariant function can be expressed using the elements of a Hilbert Basis and their gradients.



A Hilbert basis associated to the action of the group $O(2)$ (\cite[page 205]{chossat-lauterbach:00}) is given by the $O(2)$-invariant polynomials:
\[
 \pi_1 = z_1\bar z_1,\quad\pi_2 = z_2\bar z_2,\quad\quad\pi_3 = \Re\left( z_1^2\bar z_2\right) 
\]
 which must satisfy the constraints
\begin{equation}\label{eq:OSinequalities}
 \pi_1\geq 0,\quad\pi_2\geq 0,\quad\pi_3^2\leq\pi_1^2\pi_2
\end{equation}
Our analysis is now focused on the so-called Orbit Space, i.e. the subset of $\R^4$ of the four-tuples $(v_0,\vec\pi)$, $\vec\pi=(\pi_1,\pi_2,\pi_3)$, that satisfy the previous inequalities.

As the fonction $B_0(v_0,z_1,z_2)$ is $O(2)$-invariant (i.e. $B_0(g \cdot (v_0,z_1,z_2))=B_0(v_0,z_1,z_2)$ for all $g$ in $G$), it is a function, noted $\tilde B_0(v_0,\vec\pi)$, of the variables $(v_0,\vec\pi)$, i.e. $B_0(v_0,z_1,z_2)=\tilde B_0(v_0,\vec\pi)$. Furthermore since  the pair $(B_1,B_2)$ is $O(2)$-equivariant (i.e. $g \cdot (B_1,B_2)(v_0,z_1,z_2)=(B_1(g \cdot (v_0,z_1,z_2)),B_2(g \cdot (v_0,z_1,z_2)))$ for all $g$ in $G$), it can be written:
\begin{equation} \left\lbrace \label{RM:inv}
    \begin{array}{l}
       -z_1+B_1(v_0,z_1,z_2) = a(v_0,\vec\pi)z_1+b(v_0,\vec\pi)\bar z_1 z_2\\
       -z_2+B_2(v_0,z_1,z_2) = c(v_0,\vec\pi)z_2+d(v_0,\vec\pi)z_1^2
       \end{array}\right.
\end{equation}
where $a(v_0,\vec\pi),b(v_0,\vec\pi),c(v_0,\vec\pi),d(v_0,\vec\pi)$ are $O(2)$-invariant functions. Notice that this implies that $B_1(v_0,0,z_2)=0$.

Using the definition of the polynomials $\pi_i$, it is possible to rewrite (\ref{eq:discrete}) only in terms of $(v_0,\vec\pi)$ (as we did in the previous section with the polar coordinates):
\begin{equation}\label{RM:OSpace}
\left\lbrace 
    \begin{array}{l}
\dot v_0=-v_0+ \tilde B_0(v_0,\vec\pi) -\theta\\
 \dot\pi_1=2a(v_0,\vec\pi)\pi_1+2b(v_0,\vec\pi)\pi_3\\
      \dot\pi_2=2c(v_0,\vec\pi)\pi_2+2d(v_0,\vec\pi)\pi_3\\
      \dot\pi_3=\left[ 2a(v_0,\vec\pi)+c(v_0,\vec\pi)\right] \pi_3+2c(v_0,\vec\pi)\pi_1\pi_2+d(v_0,\vec\pi)\pi_1^2
       \end{array}\right.
\end{equation}
These equations are solved to find the tuning curves. 


\section*{Results}
\SOS{The results section should provide details of all of the experiments that are required to support the conclusions of the paper. There is no specific word limit for this section, but details of experiments that are peripheral to the main thrust of the article and that detract from the focus of the article should not be included. The section may be divided into subsections, each with a concise subheading. Large datasets, including raw data, should be submitted as supporting files; these are published online alongside the accepted article. The results section should be written in the past tense}
We used the equations of the previous paragraph to find the tuning curves in the cases $N=1$ and $N=2$. Some of these tuning curves corresponded to a neuronal illusion. We then showed that adding more modes to the connectivity function did not change the results. Finally we designed two different types of external stimuli for bringing the network to the illusory states.
\subsection*{Finding the tuning curves, case $N=1$}
\label{section:seuil}

%
\SOS{Ceci ne correspond a la definition d'une tuning curve}
The Ring Model is based on one main ingredient: at null contrast and for small values of the nonlinear gain $\lambda$, there is a unique stationary solution, which is not tuned! Indeed, this stationary solution has to satisfy $\rho^f=0$ otherwise it would not be unique because of the group $O(2)$ equivariance. Thus, in order to produce tuning curves (that are tuned by definition), we need a solution to \eqref{eq:onemodestatic} satisfying $\rho^f\neq 0$. This means that we must investigate for which values of $\lambda$, if any, the $\rho$ solution of \eqref{eq:onemodestatic} bifurcates. For no  external input they read:
\begin{equation}\label{eq:onemodebif}
\left\{
    	\begin{array}{lcl}
    	v_0^f&=&\varepsilon_0 B_0(v_0^f,\rho^f)-\theta \\
    	\rho^f&=&\sqrt{J_1}B_1(v_0,\rho^f)\\
    	&&\varphi^f \in\mathbb R
	\end{array}
\right.
\end{equation}
A bifurcated solution arises when
\[
1=\sqrt{J_1}\partial_\rho B_1(v_0,\rho)_{|\rho=0}=\lambda\frac{J_1}{\pi}\int\limits_{-\frac{\pi}{2}}^\frac{\pi}{2}S'(\lambda(v_0+0\cdot\sqrt{J_1}\cos_2 ))\sin_2^2=\lambda\frac{J_1}{2}S'(\lambda v_0).
\]
It follows that the equations for the existence of a tuned stationary solution, satisfying $\rho^f\neq 0$ are:
\begin{equation}\label{eq:seuil}
 \left\{
    	\begin{array}{lcl}
    	v_0^f&=&\varepsilon_0S(\lambda v_0^f)-\theta \\
    	1&=&\lambda S'(\lambda v_0^f) \frac{J_1}{2}
	\end{array}
\right.
\end{equation}
Using the relation $S'=S(1-S)$ it is straightforward to show that $\lambda$ and $J_1$ must satisfy the condition $\lambda J_1 \geq 8$. In (see \textit{Supporting information}, third section), we find other equivalent conditions that are used to produce the graphs shown 
in figure \ref{fig:condtc}. These graphs show that the threshold $\theta$ and the ratio excitation/inhibition are constrained in order to produce the tuning curves.
\begin{figure}[htbp]
\centering
          \includegraphics[width=8cm,height=7cm]{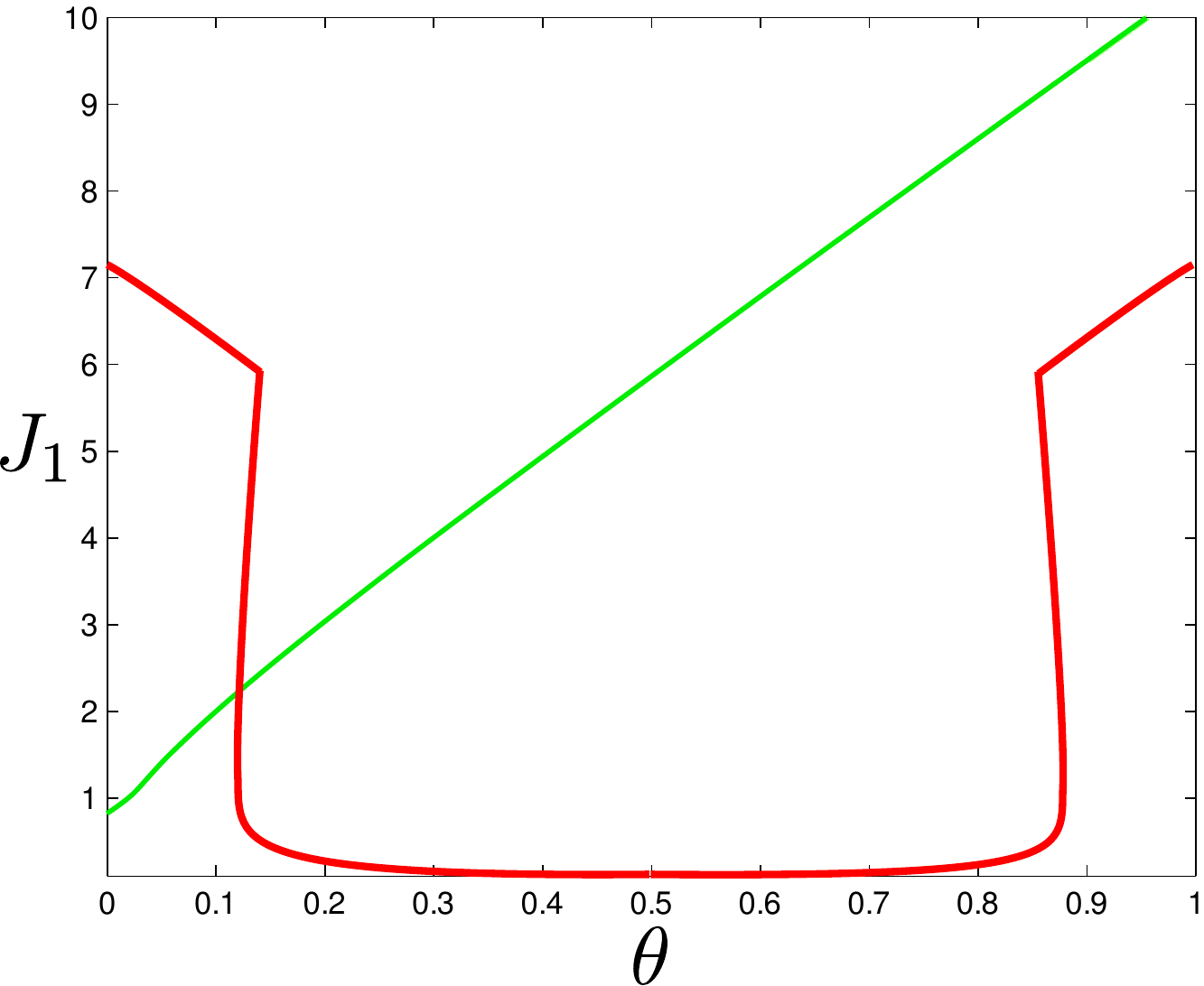}
          \caption{Boundaries in the $(J_1,\,\theta)$ plane for the existence of the tuning curves for $\varepsilon_0=-1$ (Green, it is very close to $J_1=10\theta+1$), and $\varepsilon_0=1$ (Red). The domain of existence lies above these boundaries.}
\label{fig:condtc}
\end{figure}
When these conditions are satisfied we obtain a continuum of tuning curves parametrized by the phase angle $\varphi$, noted $TC_\varphi$, which are given by
 \[
  TC_\varphi(x)=S\left[\lambda \left(v_0^f+\lambda\sqrt{J_1}\rho^f\cos_2(x-\varphi) -\theta\right)\right] 
 \]

Note that these tuning curves are dynamically stable because they are produced by a pitchfork bifurcation, as can be seen by examing equation \eqref{eq:onemodebif}.  The bifurcated branch of interest is the one corresponding to $\rho^f\geq0$.

The next step is to investigate what happens when we switch on the LGN drive, \textit{i.e.} when $\varepsilon\neq 0$.
First, when the anisotropy $\beta$ of the LGN input is not zero, the equations \eqref{eq:discrete} are not $O(2)$-equivariant anymore. This is a symmetry breaking and, as mentioned above, there are a finite number of tuning curves. Two important questions are 1) how many of the (continuum of) tuning curves remain solutions and 2) what is their stability?
For very small $\varepsilon\neq 0$, switching on the LGN can be viewed as a perturbation of the nonlinear equations when $\varepsilon=0$, as a consequence, we expect an opening of the pitchfork as we explained in \cite{veltz-faugeras:09}. This is confirmed by figure \ref{fig:rho+-}.

We know from our previous analysis that these solutions satisfy:
 
\[ 
\left\{
    	\begin{array}{l}
    	v_0^f=\varepsilon_0B_0(v_0^f,\rho^f)+\varepsilon(1-\beta) -\theta \\
    	\rho^f=\sqrt{J_1}B_1(v_0^f,\rho^f)+ \frac{\varepsilon\beta}{\sqrt{J_1}}\cos_2(\varphi^f)\\
    	2\varphi^f = k\pi,\ k\in\mathbb Z
	\end{array}
\right.
\]
Considering the two cases $k$ even and $k$ odd we obtain:
\[
\begin{array}{ccc}
	\left\{
 	\begin{array}{l}
    	v_0^f=\varepsilon_0B_0(v_0^f,\rho^f_e)+\varepsilon(1-\beta)  -\theta+\frac{\varepsilon_0}{2}  \\
    	\rho^f_e=\sqrt{J_1}B_1(v_0^f,\rho^f_e) +\frac{\varepsilon\beta}{\sqrt{J_1}} \\
    	\varphi^f_e = k\pi,\ k\in\mathbb Z
	\end{array}\right.
	&\text{ or }
	&
	\left\{	
	\begin{array}{l}
    	v_0^f=\varepsilon_0B_0(v_0^f,\rho^f_o)+\varepsilon(1-\beta) -\theta+\frac{\varepsilon_0}{2}  \\
    	\rho^f_o=\sqrt{J_1}B_1(v_0^f,\rho^f_o) -\frac{\varepsilon\beta}{\sqrt{J_1}} \\
    	\varphi^f_o = (2k+1)\frac{\pi}{2},\ k\in\mathbb Z
	\end{array}\right.
\end{array}
\]
Because $\rho\to B_0(v_0 ,\rho)$ is even and $\rho\to B_1(v_0 ,\rho)$ is odd, necessarily $\rho^f_e=-\rho^f_o$. We solve these equations for $(v_0,\,\rho_e)$ as functions of $\lambda$ by using a continuation algorithm described in \cite{veltz-faugeras:09}, the results are shown in figure \ref{fig:rho+-}.
\begin{figure}[htbp]
\centering
          \includegraphics[width=0.5\textwidth]{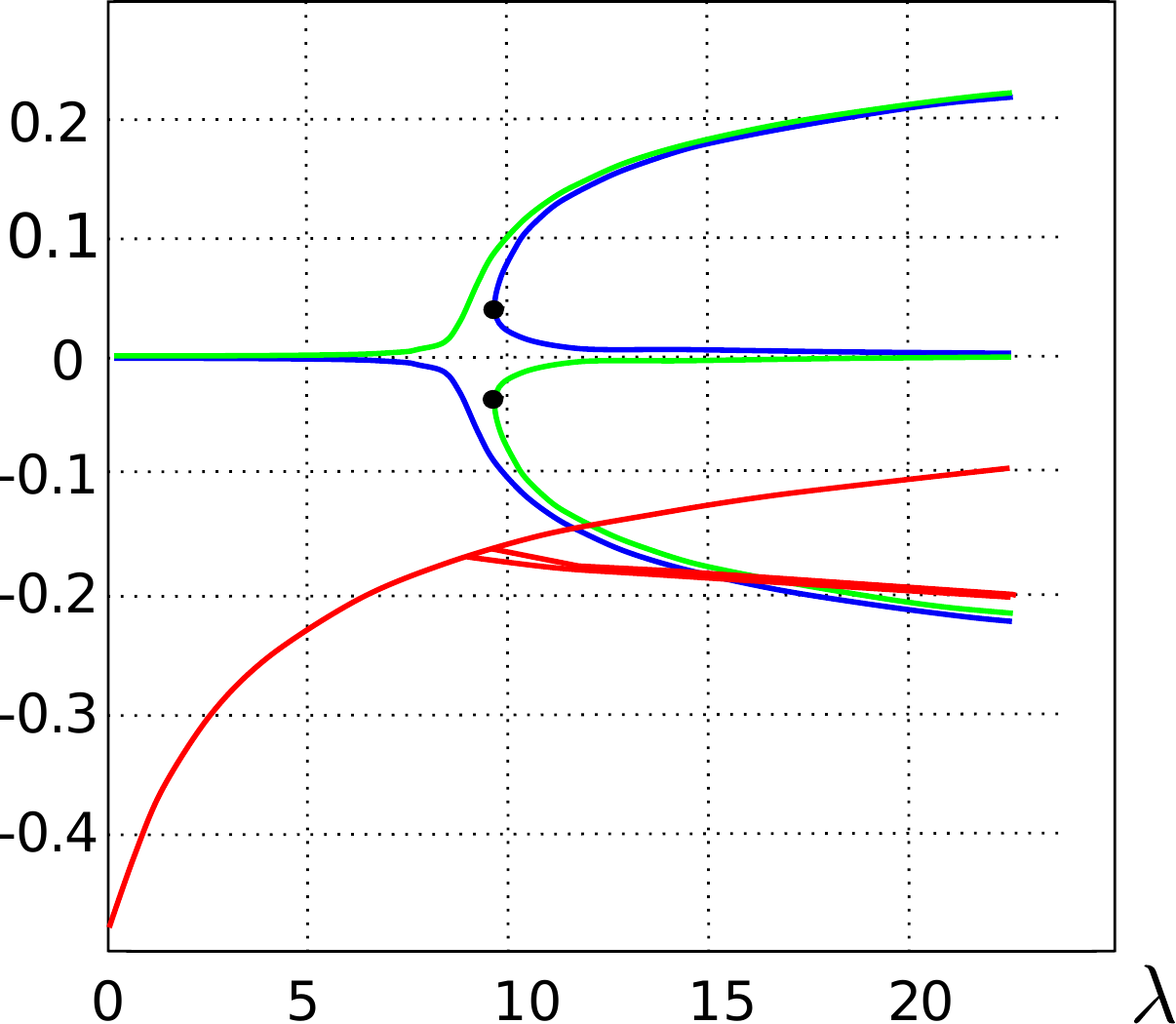}
          \caption{Plot of $(v_0,\rho_e,\rho_o)$ shown in red, green, blue, respectively, as functions of $\lambda$ for $\varepsilon=0.01,\ \theta=0,\ J_1=1.5,\ \beta=0.1$. Notice the turning points labelled with black dots.}
\label{fig:rho+-}
\end{figure}

From this bifurcation diagram, we observe that there are three stationary solutions for $\lambda>10$,  one unstable  corresponding to a small value of $\rho$ (thus it is untuned and, by definition, does not represent a tuning curve) and two others which are tuning curves. One, noted $TC_0$, is peaked at $x=0$ and the other, noted $TC_{\pi/2}$, is peaked at $x=\pi/2$. Note that the values $\lambda>10$ are in agreement with the previously derived necessary condition $\lambda J_1\geq 8$. The two tuning curves are shown in figure \ref{fig:TC}.
\begin{figure}[htbp]
	\centering
	\includegraphics[width=0.5\textwidth]{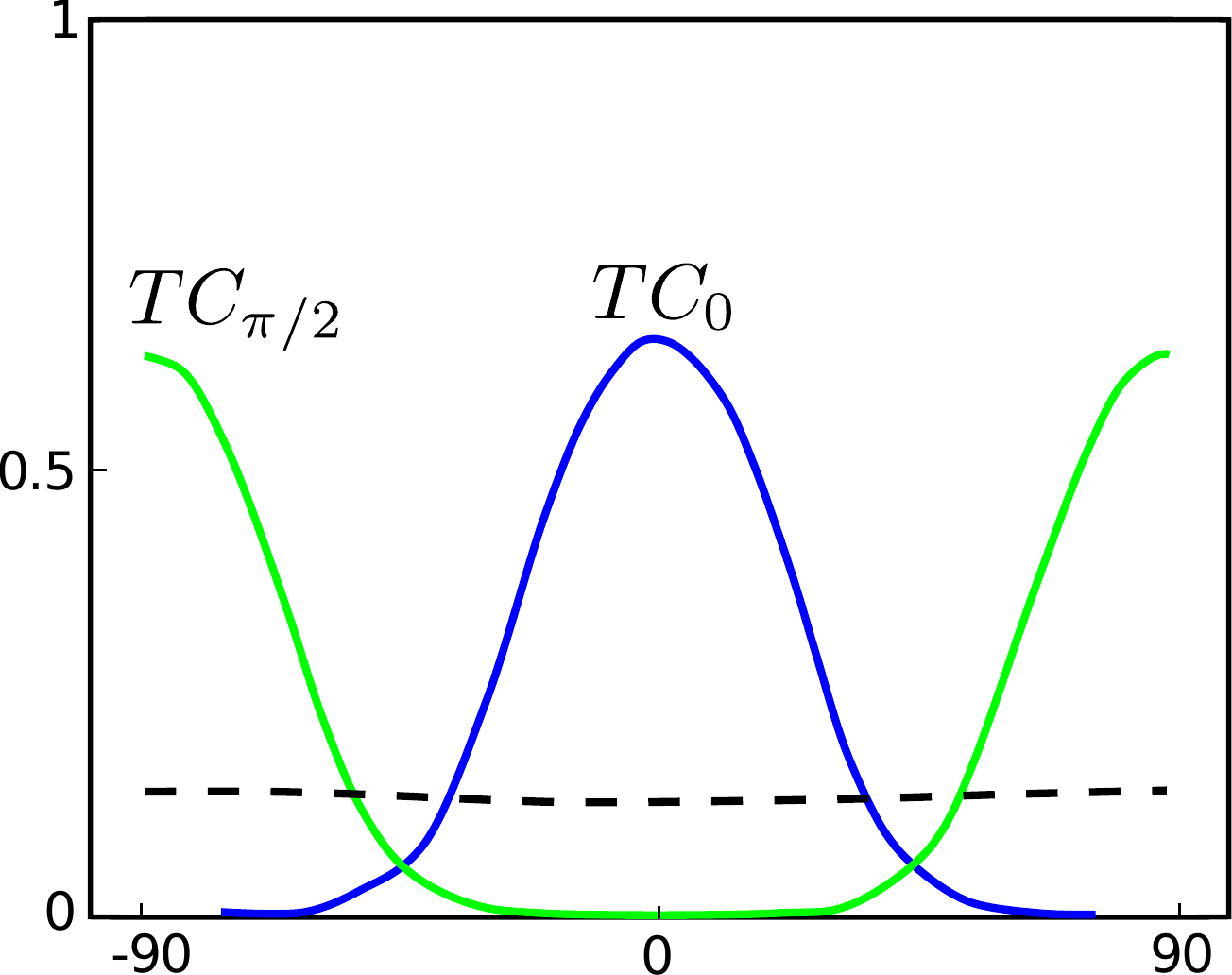}
	\caption{The tuning curves $TC_0$ and $TC_{\pi/2}$ for $\lambda=15$ and $\varepsilon=0.01,\ x_0=0,\ \theta=0,\ J_1=1.5,\ \beta=0.1$.}
	\label{fig:TC}
\end{figure}
The interesting fact to notice is that the tuning curve $TC_{\pi/2}$ is a somewhat bizarre stable state of the Ring Model. We may want to call it a neuronal illusion, the 90 degrees illusion, since it corresponds to the fact that, even if the thalamic input is peaked at the zero degree orientation, the Ring Model may be (and stay) in a stable state corresponding to a tuning curve peaking at 90 degrees! In other words, even if the thalamic input ``says'' 0 degrees, the hypercolumn of orientation ``says'' 90 degrees.
\subsection*{Finding the tuning curves, case $N=2$}
The previous results may seem to depend very much on the type of  simple connectivity function that we have assumed. In fact this is not so. By adding one more mode to this function we show that they are generic if the resulting function preserves the structure of the local excitation and the lateral inhibition.

\

The tuning curves are now solution of the nonlinear equations:
\begin{equation}
\label{eq:tcn2}
\left\lbrace 
    \begin{array}{l}
       v_0=\tilde B_0(v_0,\vec \pi)-\theta\\
       0=2a(v_0,\vec\pi)\pi_1+2b(v_0,\vec\pi)\pi_3\\
      0=2c(v_0,\vec\pi)\pi_2+2d(v_0,\vec\pi)\pi_3\\
      0=\left[ 2a(v_0,\vec\pi)+c(v_0,\vec\pi)\right] \pi_3+2c(v_0,\vec\pi)\pi_1\pi_2+d(v_0,\vec\pi)\pi_1^2
       \end{array}\right.
\end{equation}
We first look at the case where the sigmoid function $S$ is zero at the origin, i.e.  $S_0(x)=S(x)-\frac{1}{2}$. 
\begin{figure}[htbp]
\centerline{
	  \includegraphics[height = 5cm]{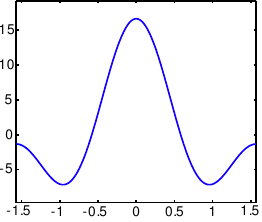}
          \includegraphics[height = 5cm]{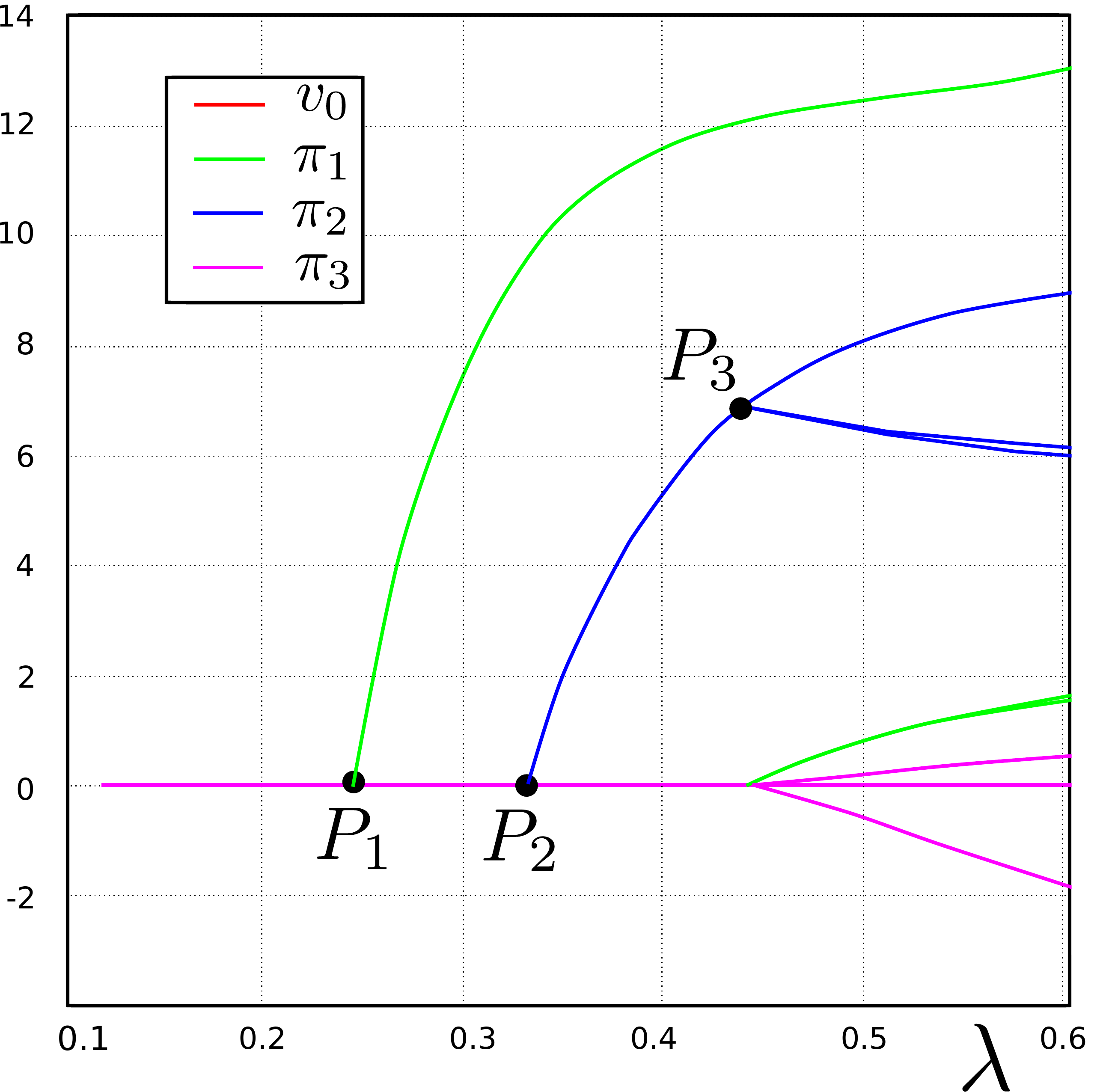}
}
          \caption{Left: Connectivity function used in the example described in the text. Right: Plot of the solutions obtained with $S_0$ instead of $S$ as functions of the nonlinear gain $\lambda$. Each solution is made of a 4-tuple (red, green,blue,violet). The parameter values are $J_0=-1,J_1=9,J_2=6.66,\ \theta=0,\ \varepsilon=0$. The 'wavy' branches are numerical artefacts of the approximation of the centered sigmoid $S_0$ by polynomials. Note that, as expected because we are on the Orbit Space, the values of $\pi_1$ and $\pi_2$ are positive, see text.}
\label{fig:OSmu0}
\end{figure}

We see on the graph of solutions on the Orbit Space shown in figure \ref{fig:OSmu0} that there are two bifurcations\footnote{These are not regular bifurcations because we are working on the Orbit Space.}  from the trivial solution $(v_0,\vec\pi)=0$, at the points, noted  $P_1$ and $P_2$ in this figure, corresponding to the values $\lambda_1 < \lambda_2$ of the nonlinear gain. 
Considering again figure \ref{fig:OSmu0} it motivates the following remarks:
 \begin{enumerate}
  \item The first bifurcated branch from $P_1$ reaches high values well before $P_2$.
  \item Our Orbit Space reduction procedure allows us to compute numerically such secondary bifurcation points as $P_3$  which might produce linearly stable tuning curves. These are undesirable from a biological viewpoint because they produce stable multimodal tuning curves.
 \end{enumerate}

The linear stability analysis shows that the branch bifurcating from $P_1$ is stable and corresponds to a continuum of tuning curves parametrized by the phase qngle $\varphi$ and given by:
\[
 \forall\varphi\quad TC^1_\varphi(x)=S_0\left[\lambda \left(v_0^f+\sqrt{\pi_1^fJ_1}\cos_2(x+\varphi)-\theta\right)\right] 
\]
The unstable tuning curves bifurcating from $P_2$ (before $P_3$) are given by: 
\[
 \forall\varphi\quad TC^2_\varphi(x)=S_0\left[\lambda \left(v_0^f+ \sqrt{\pi_2^f|J_2|}\cos_4(x+\varphi)-\theta\right)\right] 
\]
This shows that in order to have unimodal tuning curves (responses of the hypercolumn represented by the Ring Model), it is necessary that $\lambda_1<\lambda_2$ or equivalently $J_2<J_1$ because the nonlinear gains which produce the pitchforks are $S'_0(0)\lambda_i=\frac{2}{J_i}$, $i=1,2$.  Moreover, since $\pi_1^f$ quickly reaches high values, $TC^1_0(0)$ is close to $1$: the response does not depend  upon the contrast $\varepsilon$ of the LGN. \textit{This implies that the working range of the nonlinear gain $\lambda$ is close to the value $\lambda_1$.}

\

It is now possible to understand the diagram of solutions shown in figure \ref{fig:OSmu1} obtained with the regular sigmoid $S$ as a deformation of the diagram shown in figure \ref{fig:OSmu0}. As in the case $N=1$, the bifurcated branches will persist if the coefficients $J_i$, $i=1,2$, and the threshold $\theta$, satisfy the constraints shown in figure \ref{fig:condtc}. Indeed, if we reproduce the analysis in the paragraph \textit{Materials and Methods}, the existence of a pitchfork for the $z_i$, $i=1,2$ coordinate  is given by 
 \[
 \left\{
    	\begin{array}{lcl}
    	v_0^f&=&\varepsilon_0S(\lambda v_0^f)-\theta \\
    	1&=&\lambda S'(\lambda v_0^f) \frac{\textstyle J_i}{\textstyle 2}
	\end{array}
\right.
\]

We again notice that the first branch bifurcating from $P_1$ (in green in figure \ref{fig:OSmu1}) is quickly reaching high values and that the tuning curve is now asymmetric (this is much easier to see in the middle part of figure \ref{fig:TCmu1}). This is because the $\pi_2,\pi_3$ components (in blue and magenta in figure \ref{fig:OSmu1}) are not zero as in the case $N=1$.  The tuning curve corresponding to the first bifurcated branch  is given 
\begin{equation}\label{eq:tuningN2}
 TC(x)=S\left[\lambda \left(v_0^f+ \sqrt{\pi_1^fJ_1}\cos_2(x)+ \underset{small}{\underbrace{ \sqrt{\pi_2^f|J_2|}}}\cos_4(x+\varphi_2-\varphi_1)-\theta\right)\right] 
\end{equation}
where $z_1=\sqrt{\pi_1}e^{2i\varphi_1},\ z_2=\sqrt{\pi_2}e^{4i\varphi_2}$, and $\pi_1\sqrt{\pi_2}\cos_4(\varphi_2-\varphi_1)=\pi_3$.
Remember that there is an infinity of tuning curves that are obtained from the one given by equation \eqref{eq:tuningN2} by applying an element of the group G.

\begin{figure}[htbp]
\centering
          \includegraphics[scale = 0.25]{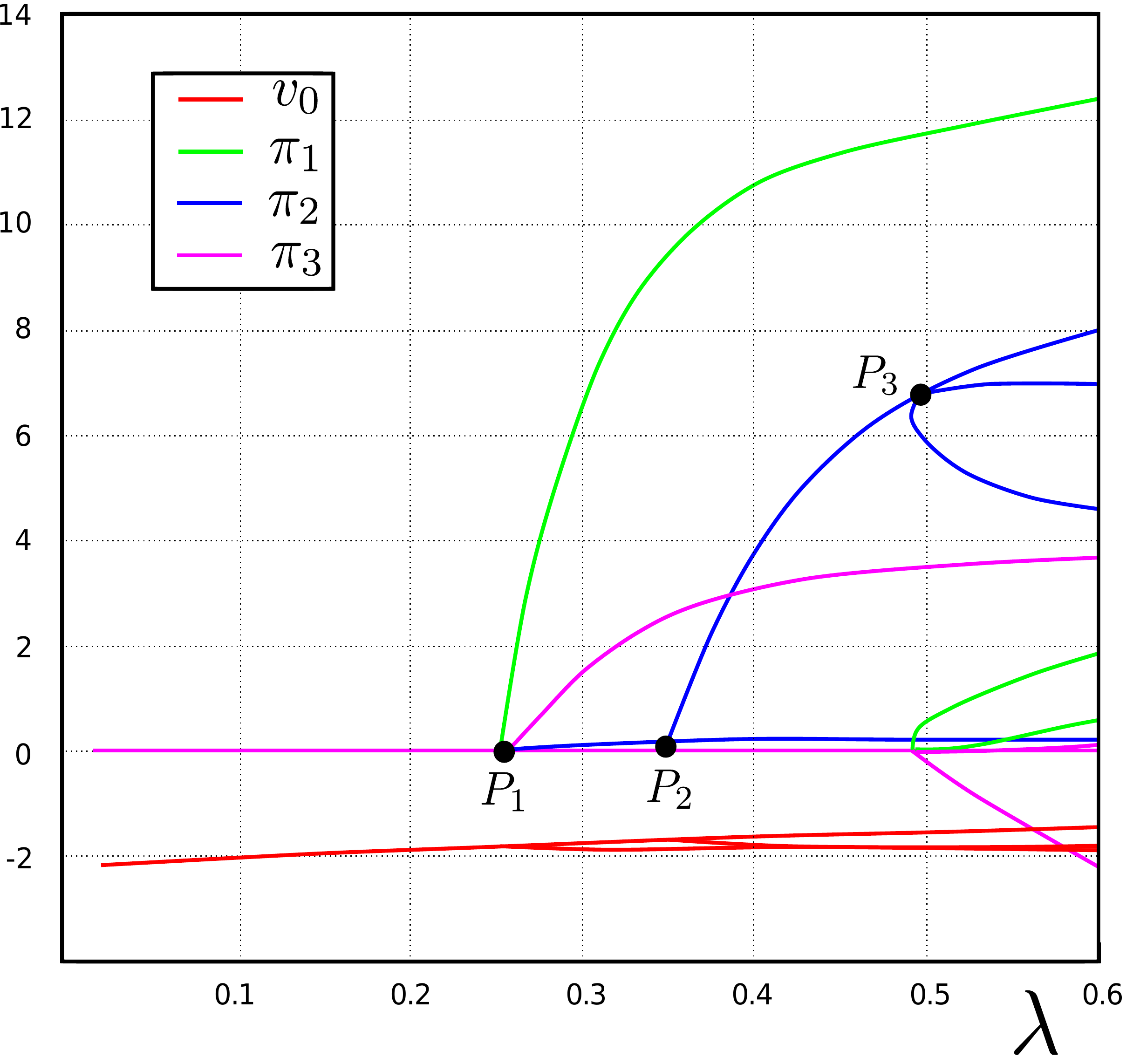}
          \caption{Plot of the solutions with the regular sigmoid $S$ as a function of the nonlinear gain $\lambda$. Case $J_0=-1,J_1=9,J_2=6.66,\ \theta=0\ \varepsilon=0$.}
\label{fig:OSmu1}
\end{figure}
We have plotted in figure \ref{fig:TCmu1} examples of the tuning curves for three values of the nonlinear gain $\lambda$ that are slightly larger than the values corresponding to the three bifurcation points $P_1$, $P_2$ and $P_3$ in figure \ref{fig:OSmu1}. These tuning curves are obtained by reading from figure \ref{fig:OSmu1} the 4-tuple $(v_0,\vec{\pi})$. This yields, through the relation $\pi_1\sqrt{\pi_2}\cos_4(\varphi_2-\varphi_1)=\pi_3$, the value of $\varphi_2-\varphi_1$ that is needed in equation \eqref{eq:tuningN2}. Notice the unstable multimodal tuning curves that appear once the stable tuning curve has saturated (middle plot in figure \ref{fig:TCmu1}). This is an indication that the nonlinear gain should not be too high otherwise most responses of the network will be saturated. 

\begin{figure}[htbp]
 \centerline{
          \includegraphics[width=0.3\textwidth]{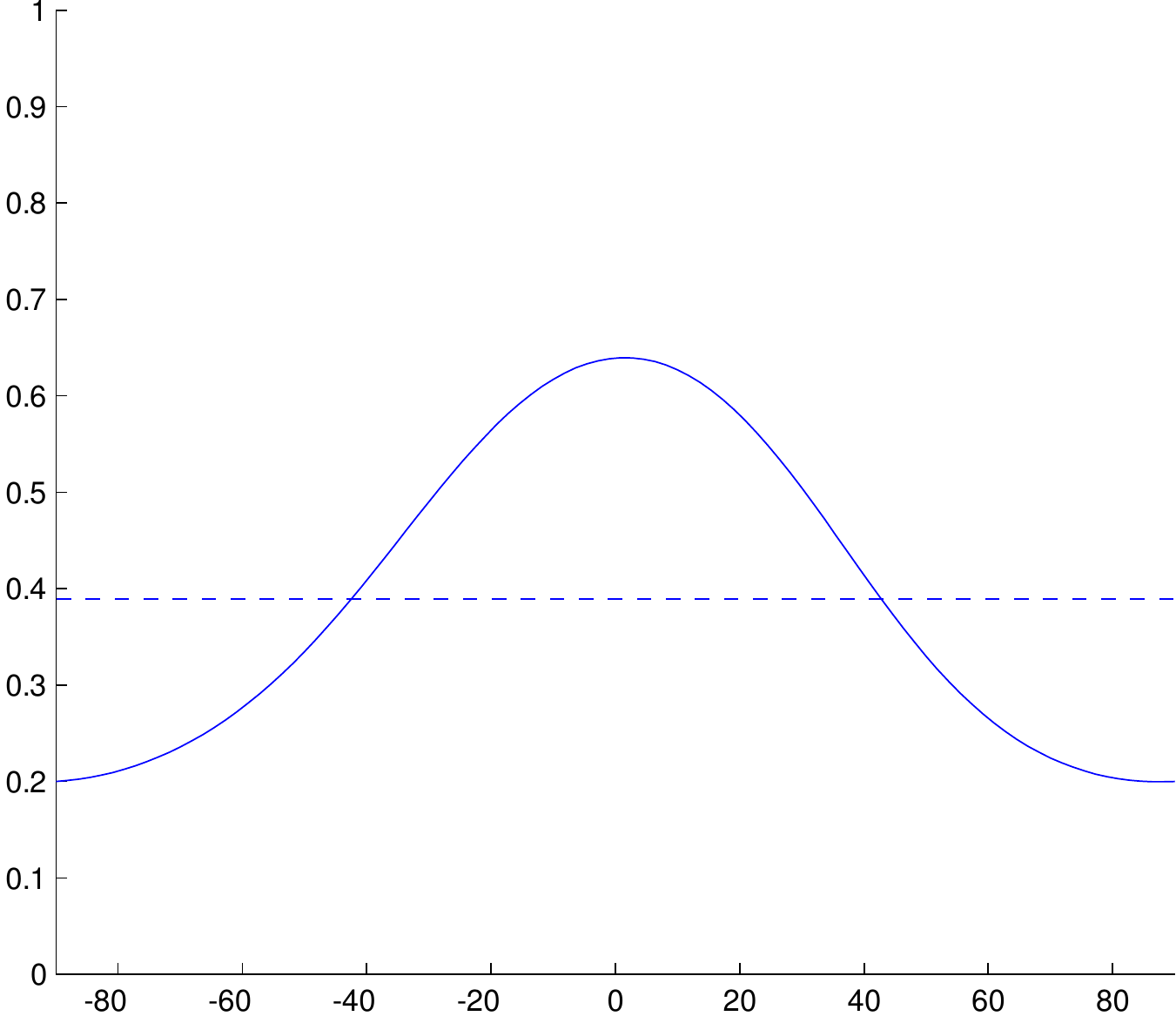}
	  \includegraphics[width=0.3\textwidth]{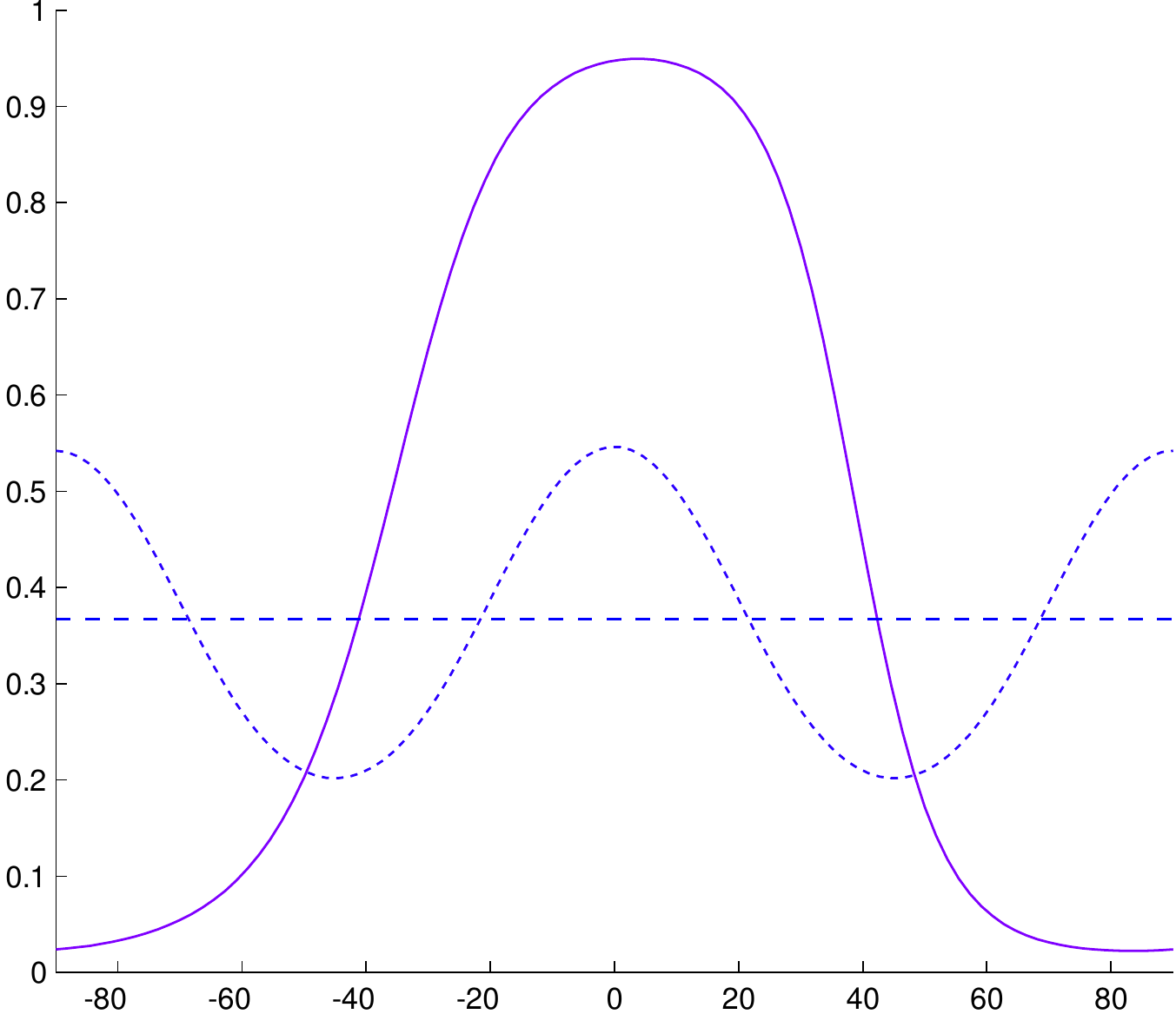}
	  \includegraphics[width=0.3\textwidth]{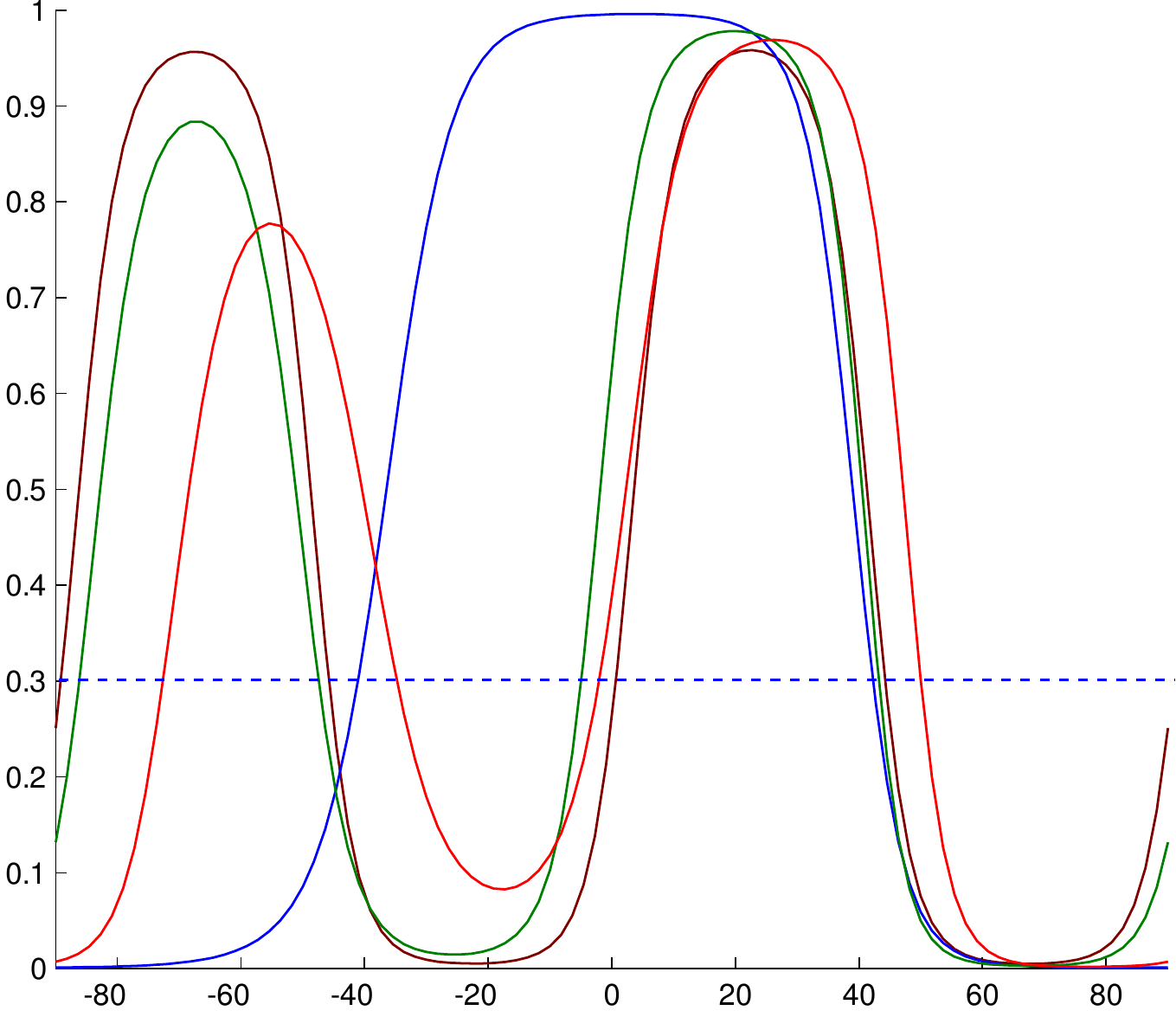}  
}
          \caption{Tuning curves at slopes $\lambda=0.26$ (left),\ $\lambda=0.355$ (middle), and  $\lambda=0.58$ (right) when the input is equal to 0, see text. On the left and in the middle, stable tuning curves are shown in continuous line, unstable ones in dotted lines. Stability is not shown in the plot on the right, except for the null solution. The other parameters are the same as in figure \ref{fig:OSmu1}.}
          \label{fig:TCmu1}
\end{figure}

If we switch on the LGN, the contrast $\varepsilon$ is nonzero. The external current is given by $I(x)=I_0+I_1\sqrt{J_1}\cos_2+I_2\sqrt{|J_2|}\cos_4$. If $I_1I_2\neq 0$, the symmetries of the equation \eqref{RM:OSpace} are broken and we expect a finite number of solutions. More precisely, the same argument as for $N=1$ shows that we can interpret \eqref{eq:discrete} as a perturbation of \eqref{RM:OSpace} when $\varepsilon$ is small, leading to an opening of the pitchforks. Hence, when the nonlinear gain $\lambda$ is close to that of $P_1$ we have $z_2\approx 0$ for $\varepsilon$ small. Since the equations 
\begin{equation*}\left\lbrace 
\begin{array}{l}
-v_0+B_0(v_0,z_1,z_2)-\theta+\varepsilon I_0=0\\
 -z_1+B_1(v_0,z_1,z_2)+\varepsilon I_1=0
\end{array}\right.
\end{equation*}
are the same as in the case $N=1$ when $z_2=0$, they do not change much when $z_2\approx0$ and the 90 degrees illusion found in the previous section remains: there are two tuning curves, one peaking as the external input $I$ and one translated by 90 degrees. This analysis is confirmed by the results of the numerical computations shown in figure \ref{fig:TCN2} where we show the solutions of \eqref{eq:discrete} for $N=2$, $I_0=1-\beta,\ \sqrt{J_1}I_1=\beta,\ \sqrt{|J_2|}I_2=0.1\beta$.
\begin{figure}[htbp]
	\centering
	\includegraphics[width=0.3\textwidth]{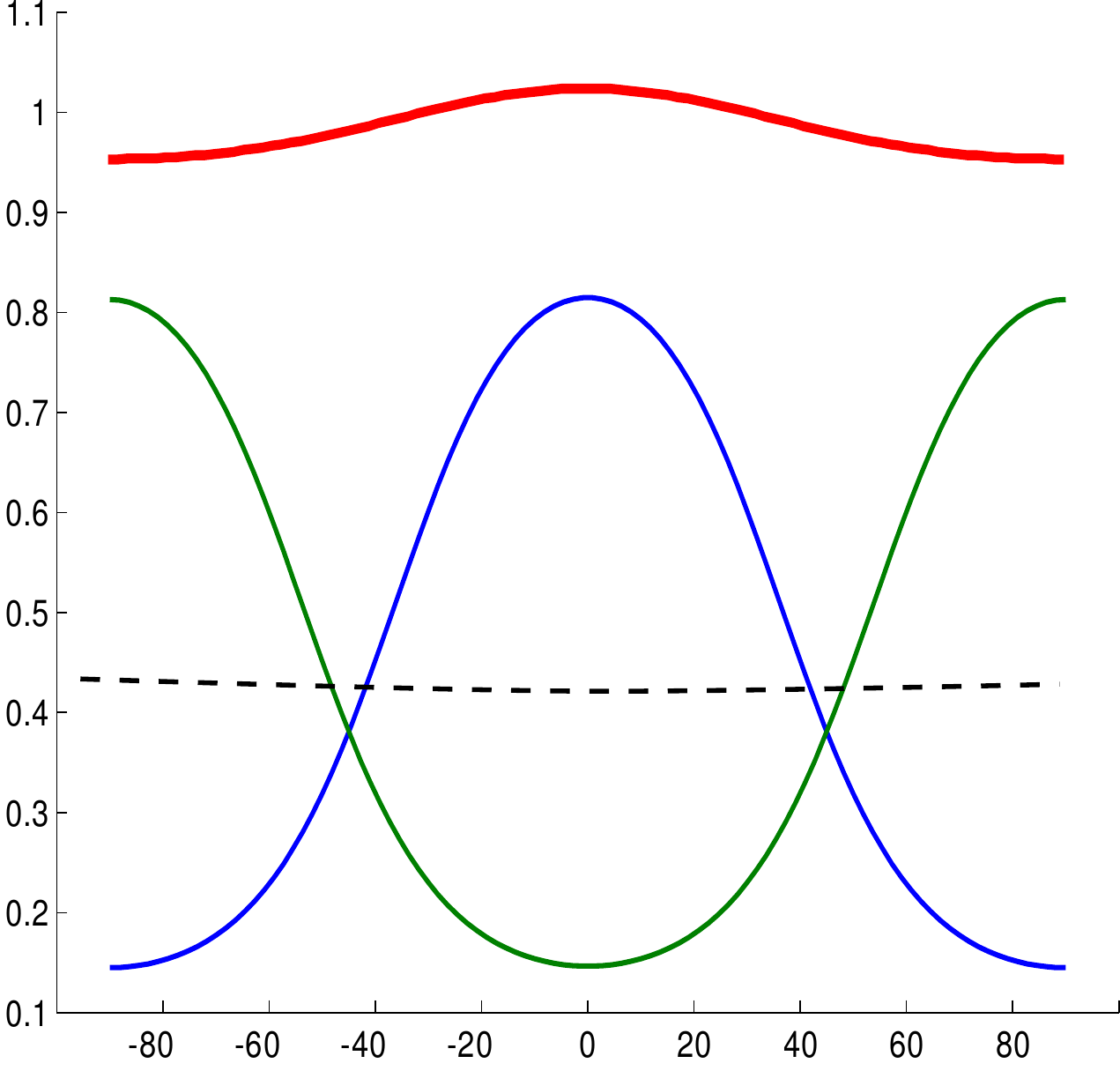}
	\caption{The tuning curves $TC_0$ and $TC_{\pi/2}$ for $\lambda=0.26$, $\varepsilon=0.01,\ x_0=0,\ \theta=0,\ J_1=9,\ J_2=6.66,\ \beta=0.05$ and $I_0=1-\beta,\ \sqrt{J_1}I_1=\beta,\ \sqrt{|J_2|}I_2=0.1\beta$. $I$ is plotted in red. Notice the unstable weakly tuned tuning curve shown in black.}
	\label{fig:TCN2}
\end{figure}
\subsection*{Arbitrary number of modes}
We can perform the same computations using more modes, this will only bring in more tuning curves. 
Because these tuning curves only appear once the stable unimodal tuning curve has saturated these high values for the nonlinear gain $\lambda$ are biologically irrelevant or nonplausible. Notice also that the neuronal illusions found in the case $N=1,\,J_1>0$ are still present for $N > 1,\,J_1>0$, as shown for example in figures \ref{fig:TCmu1} and \ref{fig:TCN2}.
Indeed, as seen in the previous section, they only depend upon the fact that the network features a pitchfork bifurcation at the point noted $P_1$ in figures \ref{fig:OSmu0} and \ref{fig:OSmu1} and this is always the case for any value of the number $N$ of modes if the coefficients $J_i$ satisfy the mild constraints we have described previously and we summarize in the next section.
\subsection*{ Dynamical 90 degrees illusions}
In the last paragraphs, we found two cortical representations of the same external stimulus. An immediate question is how can we bring a hypercolumn of orientation into each of these two states? Can we drive the cortical state to the illusion using only the stimulus $I$? We answer this question positively in the next two paragraphs.
\subsubsection*{Rotating the stimulus back and forth}
As the illusory tuning curve $TC_{\pi/2}$ is very close to the cortical state corresponding to the response of the network to a stimulus peaked at $\pi/2$, we first present a stimulus peaked at $0$ to put the system in the $TC_0$ state corresponding to no illusion. We then slowly change the position of the peak of the external stimulus $I$ and bring it to the value $\pi/2$. The network follows the stimulus and its response is peaked at $\pi/2$. We then suddenly change the stimulus to a stimulus peaked at $0$ and since the responses of the network to stimulus oriented at $0$ or $\pi/2$ are very close, the cortical state will remain in the state peaked at $\pi/2$, the one it is in just before the sudden change of the stimulus. This is reminiscent of the after-effect illusion and can be confirmed by numerical simulation.

The resulting effect is shown in figure \ref{fig:illusion1} when
the time variation of the peak $x_0$ of the stimulus in equation \eqref{eq:I} is given by
\begin{equation*}
 x_0(t)=\left\lbrace 
 \begin{array}{cl}
 \frac{\pi}{2}\min(\frac{t}{1000},1)&\text{ if }t\in[0, 2000]\\
  \frac{\pi}{2}&\text{ if }t\in[2000, 2e4]\\
  0&\text{ if }t>2e4
 \end{array}\right.
\end{equation*}

\begin{figure}[htbp]
\centering
          \includegraphics[width=0.6\textwidth]{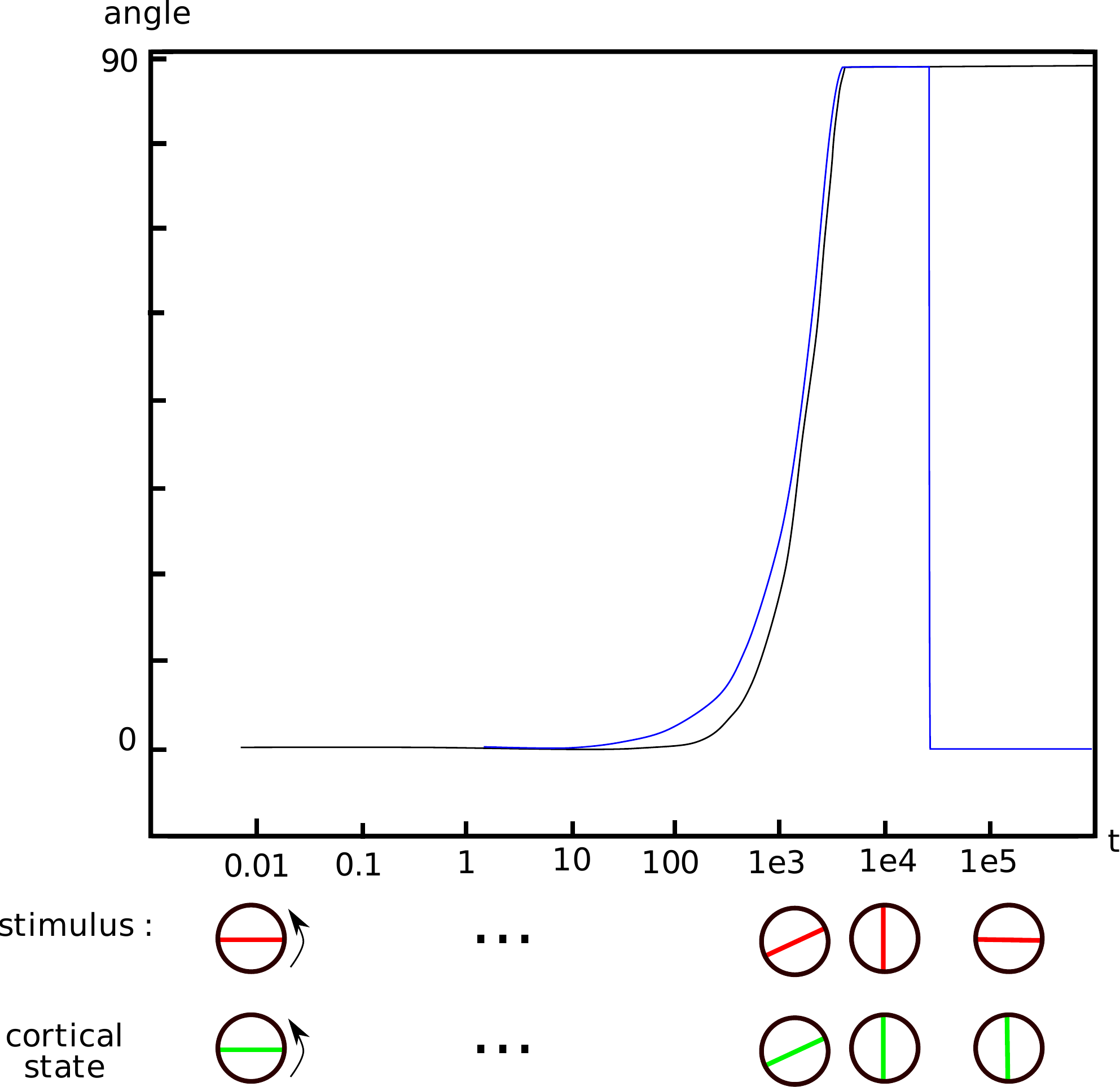}
          \caption{Plot of the position $x_0$ (shown in blue) of the peak of the stimulus $I$ and the phase coordinate $\varphi$ (shown in black) of the network, both as functions of time, on a logarithmic scale. Note that the stimulus drives the network into a state that is very close to its expected state when presented with a stimulus oriented at 90 degrees and that it stays there even after the input is switched back to a 0 orientation stimulus. The parameters are the same as in figures \ref{fig:rho+-} and \ref{fig:TC}.}
\label{fig:illusion1}
\end{figure}

\subsubsection*{Using a mixture of the two stimuli}
This second dynamical 90 degrees illusion is very close in principle to the one presented in the previous paragraph.
Instead of rotating the stimulus,  we change its contrast as follows. Let us note $I_0$ the stimulus peaked at $0$ and  $I_{\pi/2}$ the one peaked at ${\pi/2}$, the thalamic input to the hypercolumn of orientation takes the form
\[
 I(t) = (1-\psi(t))I_0+\psi(t)I_{\pi/2}
\]
where $\psi(t)$ is the function shown in figure \ref{fig:contrast}.
\begin{figure}[htbp]
\centering
          \includegraphics[width=0.5\textwidth]{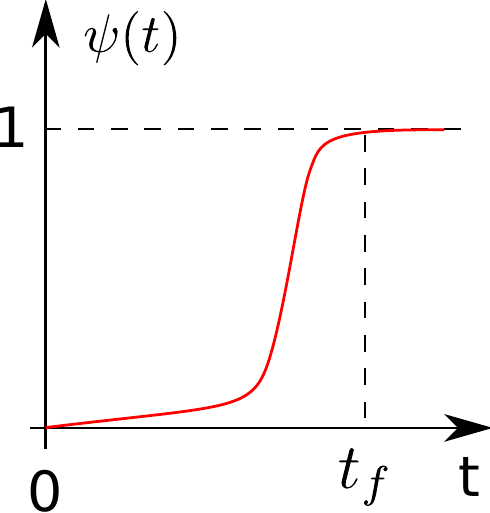}
          \caption{Plot of the function $\psi$ allowing to vary the contrast of the thalamic input, see text.}
\label{fig:contrast}
\end{figure}
We check numerically, using the dynamics given by equations \eqref{eq:onemode}, that the hypercolumn stays in the cortical state $TC_0$ (see figure \ref{fig:ill2}), despite the fact that the final stimulus corresponds to an orientation of 90 degrees.
\begin{figure}[htbp]
\centerline{
          \includegraphics[height=5cm]{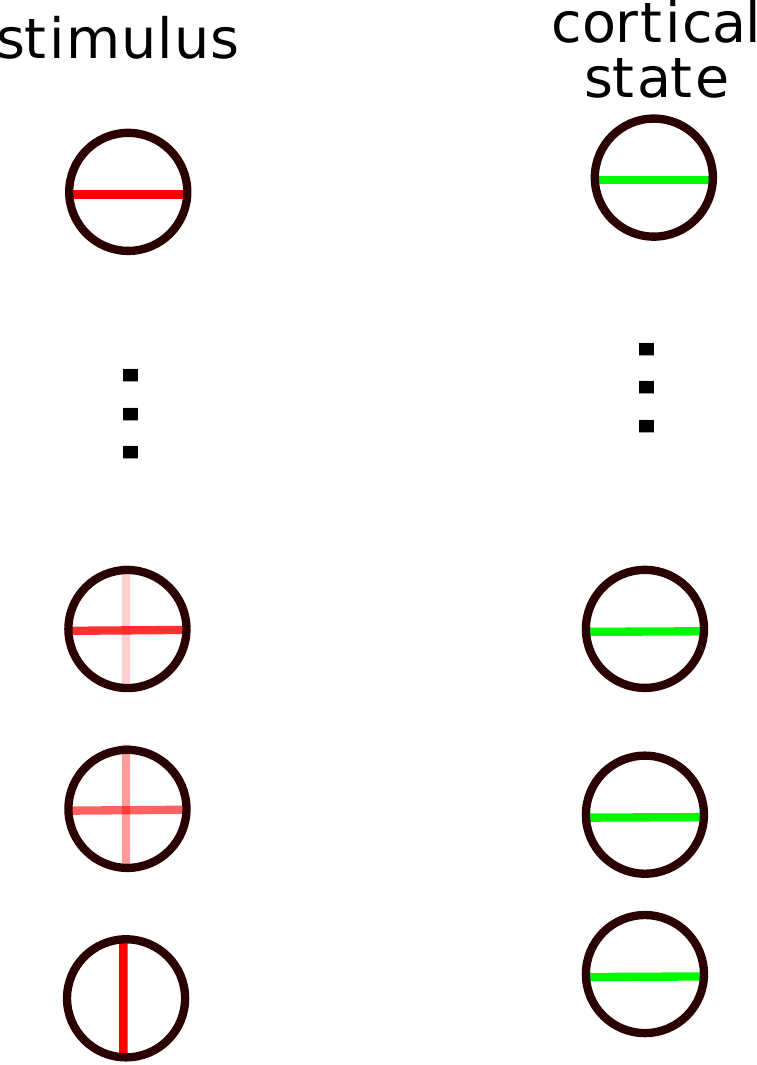}
          \includegraphics[height=5cm]{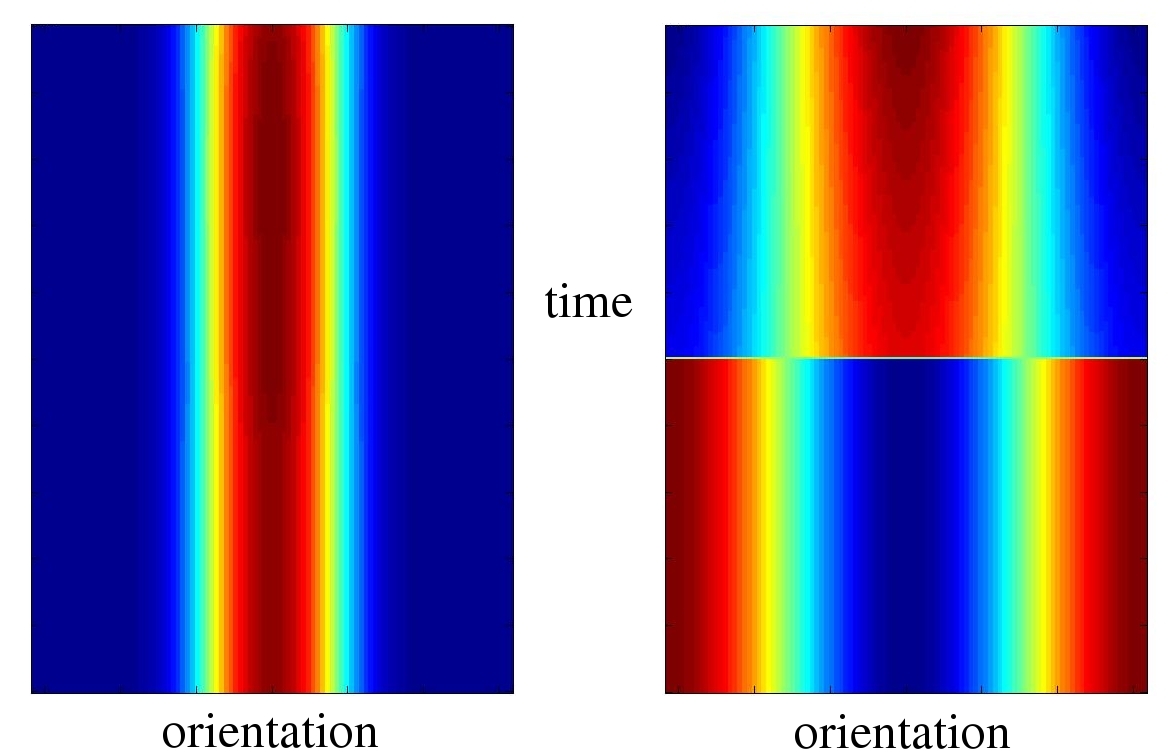}
}
          \caption{The vertical direction represents the time, the horizontal one the orientation of the stimulus and of the network response. Left: Representation of the creation of the 90 degrees illusion. Middle: The stimulus starts as $I_0$ to become a mixture of $I_0$ and $I_{\pi/2}$ and finishes as a pure $I_{\pi/2}$. Right: The response of the network shows that it stays in the cortical state $TC_0$.}
\label{fig:ill2}
\end{figure}

\section*{Discussion}
\label{section:discussion}
We now have a clear view of the functional impact of each parameter in the model. It turns out that the combination of mathematical and biological constrains basically fixes their values.
 We first notice that the requirement for unimodal responses is a very strong constraint. Indeed, it implies that the first pitchfork bifurcation that occurs when $\lambda$ varies has to be the one corresponding to the first mode and this requires
\[
 J_1\geq 0,\quad J_1>J_i\quad\forall i \neq 1
\]
Next,  in order to actually see these tuning curves, the threshold $\theta$ should not be too high: the condition  shown in figure \ref{fig:condtc}  is approximately  $J_1\geq 10\theta+1 $ for $\varepsilon_0=-1$. 
This in turn gives the range for the nonlinear gain $\lambda$: it should be high enough in order for the model to produce tuning curves but smaller than the value for which the tuning curves saturate, hence do not vary with the input contrast anymore, in contradiction with the biological measurements. This means that $\lambda\sim\frac{1}{J_1}$.
The last relevant parameter is the width of the tuning curve. As shown in \textit{Supporting information}, last paragraph, it can easily be estimated when $\pi_2$ can be neglected which is often a reasonable assumption, see figure \ref{fig:OSmu1}. It turns out to depend upon the ratio $\frac{\textstyle \theta}{\textstyle J_1}$. This closes the loop: we have set all the three parameters $(\lambda,\,J_1,\,\theta)$ and constrained the others ($J_i,\,i\geq 2$).


We now discuss the appearance of tuning curves (\textit{i.e.} stationary solutions such that $z_1\neq 0$, $|z_1|>>|z_i|,\ i>1$). What is the condition on the external input in order to produce such stationary solutions? Let us consider the case $N=1$. We have seen (see figure~\ref{fig:rho+-}) that if a tuning curve exists, then we have three stationary solutions and there is a value $\lambda_c$ of the nonlinear gain $\lambda$ (roughly equal to $10$ in figure~\ref{fig:rho+-}) at which the two tuning curves disappear, we call it a turning point (it results from an opening of the pitchfork when $\varepsilon=0$). When varying $\varepsilon$, we can look for the value of the nonlinear gain $\lambda$ (if there is one) at which a turning point occurs: it is an indication that tuning curves do exist for higher nonlinear gains. The TRILINOS package features the numerical continuation of the locus of the turning point and, starting with the turning point of \ref{fig:rho+-}, it can produce the locus of the turning points in the plane $(\varepsilon,\lambda)$ as shown in figure \ref{fig:TP}.
\begin{figure}[htbp]
\centering
          \includegraphics[width=0.5\textwidth]{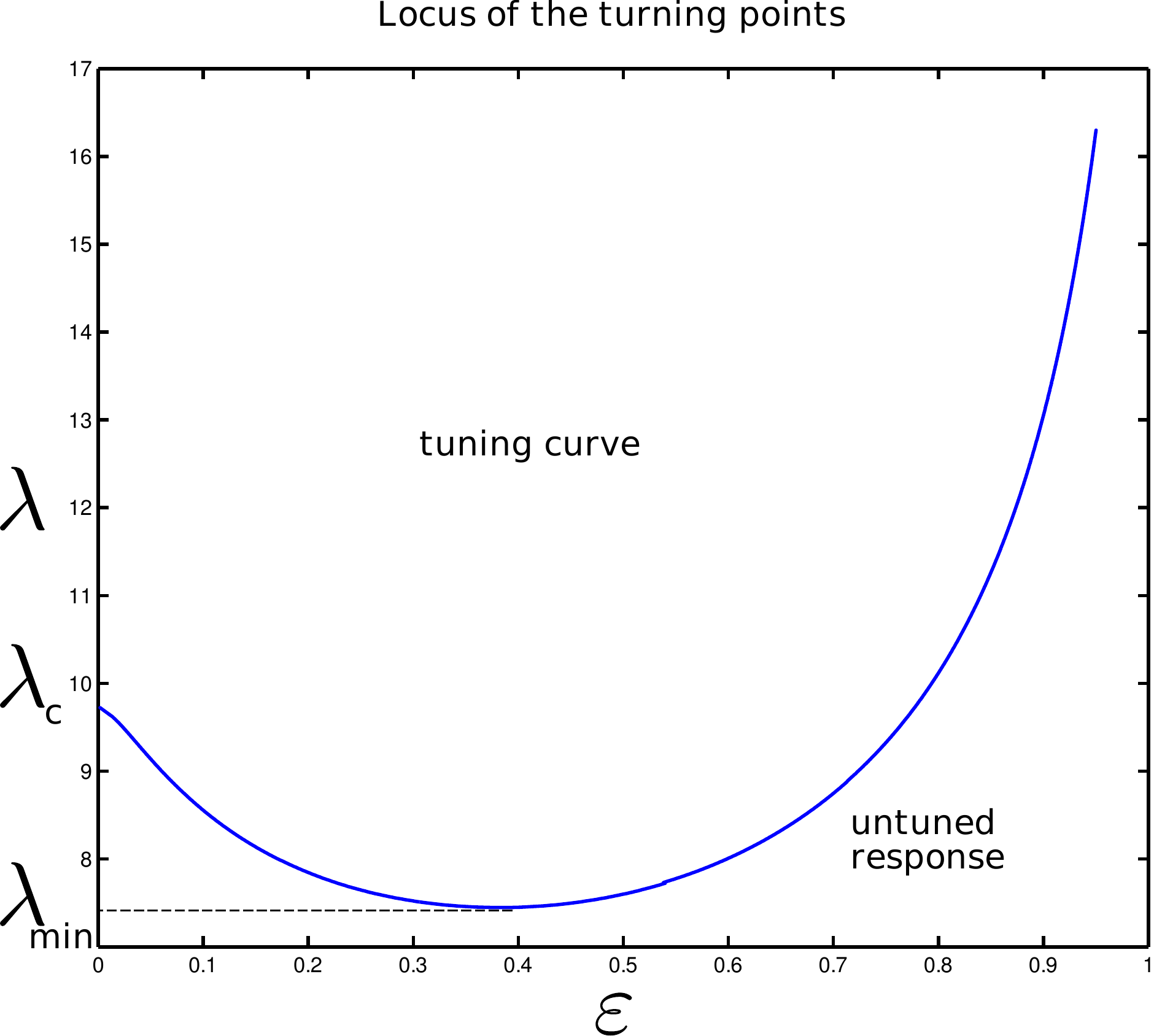}
          \caption{Locus of the turning points as function of $(\varepsilon,\lambda)$. Same parameters as in fig.\ref{fig:rho+-}.}
\label{fig:TP}
\end{figure}
Above the blue curve, the stable response of the network is a tuning curve. Hence, if $\lambda>\lambda_c$, even for no external input, $\varepsilon=0$, the network will be in a stable stationary state $V^f(x)$ which is a tuning curve whose tuning angle $\arg\max_x V^f(x)$ is randomly selected. However, if $\lambda<\lambda_c$ until the contrast $\varepsilon$ has reached a certain value, the response will be untuned. It is more biologically relevant that the network operates in the regime $\lambda<\lambda_c$, otherwise the neurons would have a high firing rate (around $60\%$ of their maximum firing rate, see figure \ref{fig:rho+-}) even though no stimulus is present. Notice that in figures~\ref{fig:rho+-}, \ref{fig:TC}, we are in the case $\lambda>\lambda_c$.
Hence, for the parameters of figure~\ref{fig:rho+-}, the working range of the nonlinear gain is $\lambda\in[\lambda_{min},\lambda_c]$. Notice that the condition on the threshold (see figure~\ref{fig:condtc}) gives the condition for $\lambda_c$ to exist (in the limit $\varepsilon\to 0$, the turning point converges to the pitchfork point). Also, $\varepsilon I_0$ lowers the threshold $\theta$ and taking the differential of \eqref{eq:seuil} w.r.t. $\theta$, it can be shown that $\frac{\partial\lambda_c}{\partial\theta}_{\theta\geq0}<0$ when $\varepsilon_0<0$. Hence we expect that the local behavior around $(0,\lambda_c)$ in figure~\ref{fig:TP} to be quite general in the case $\varepsilon_0<0$. This analysis hence provides a tighter constraint on the nonlinear gain $\lambda$, it should be just below $\lambda_c$: $\lambda\lesssim\lambda_c$.

\

The fact that the cortical network shows two states corresponding to perpendicular orientations in response to a single stimulus can also be put in resonance with some published models of the cortical primary visual area (see \cite{bressloff-cowan-etal:01} for a spatial network of Ring Models). Indeed, in this study of planforms in relation to visual hallucination, it may come as a surprise to the attentive reader that most of the planforms (in the cortical space) do not respect the good continuation principles of contours since adjacent hypercolumns show responses corresponding to orthogonal orientations. However once we agree that, for a hypercolumn, two orthogonal states are equivalent, this becomes perhaps less surprising. 


We relate our formalism to previous studies of recurrent models of orientation selectivity by first noting that
the 90 degrees illusion was not reported in \cite{ben-yishai-bar-or-etal:95} although they share the same assumptions as ours.

In \cite{carandini-ringach:97}, the authors used a (voltage based) Ring Model in order to explain some of the features of the complicated spiking network of \cite{somers-nelson-etal:95}. Although they used the non-saturating nonlinearity $S(x)=\max(x,0)$, they observed that narrowing the spatial extension of inhibition leads to mulimodal responses which they interpreted as neuronal illusions. This can be understood within our formalism: decreasing the spatial extent of inhibition introduces more Fourier terms (possibly with high values) in the connectivity $J$ and can produce multimodal responses to a unimodal stimulus (see \textit{Finding the tuning curves, case $N = 2$}). This type of nonlinearity (see also \cite{shriki-hansel-etal:03}) cannot produce the 90 degrees illusion we have described, more generally, it is not possible to produce the 90 degrees illusions using non-saturating nonlinearities. Remember that the saturation arises when one takes into account the refractory period of neurons (see \cite{gerstner-kistler:02b}), or more simply the fact that the firing rate of neurons is bounded.

Under what conditions do the 90 degrees illusions survive in a network of Ring Models? Can we find similar illusions in more sophisticated networks and which experiments could confirm/invalidate our predictions? We just discussed the matter of a network of Ring Models with the study of \cite{bressloff-cowan-etal:01}. In \cite{blumenfeld-bibitchkov-etal:06}, the authors used a generalization of the Ring Model with a very similar connectivity to explain the spontaneous activity observed in optical imaging recordings. They could not report the 90 degrees illusion because of their use of a non-saturating nonlinearity. If  a detailed mathematical analysis of a modified version of their model to include a saturating rate function were conducted, it is very likely that the 90 degrees illusion would be predicted by the modified model. 

Finally, despite its ability to reproduce several experimental facts, the Ring Model lacks some anatomical data support because it does not use realistic cortical circuitry. Recently, M. Shelley \textit{et al.} (see \cite{shelley-mclaughlin:02}) introduced a reduced system of a computationally intensive spiking neuron network model of a hypercolumn with realistic cortical circuitry. Although they do not use a refractory period in their spiking model (hence, their reduced model has a nonsaturating nonlinearity), it could be interesting to look for neuronal illusions predicted by their model using the techniques developed in \cite{veltz-faugeras:09}.

\SOS{Olivier : justement, cette figure ne montre-t-elle pas un mauvais exemple puisque $\lambda=15$ ?->[Je suis d'accord...d'un autre cote, il faut vraiment etre proche de $\lambda_c$ sinon il n'y a pas de sharpening avant une 'grosse' valeur du seuil] D'après ta figure, pour cette valeur tu es dans le mode ``toujours une tuning curve''. D'autre part il semblerait que pour $\lambda < \approx 7.5$, tu ne puisses jamais avoir de TC. Commentaires ?->[c'est lie a la condition sur le seuil, on sait grossierement que $\lambda J1>8$ (cas $J1=1.5$), mais une analyse plus detaillee du lieu de la pitchfork devrait nous donner cette borne mini $\lambda > \approx 7.5$, d'ailleurs je pense que  la fig.12 est assez proche au travail sur le seuil, cf. la discussion de cet aprem sur le sujet.]}

\SOS{Tronquer apporte des illusions, cf rmRingach.mw}

\SOS{Ancienne conclusion :}
\section*{Conclusion}

We have pushed further the study, started in \cite{veltz-faugeras:09}, of the mathematical properties of the Ring Model of orientation tuning and of some of their biological implications. This was achieved by taking into consideration the rich symmetries of the network. 

For the first time to our knowledge in the field of neural networks, we have introduced the Orbit Space Reduction technique to deal with translation invariant connectivity kernels. This allowed us to find a suitable change of coordinates in order to remove the redundancy introduced by the symmetries. This is a generic technique that can be applied to many other problems in neuroscience. 
Using this reduction, we have shown that the exact shape of the connectivity function did not matter much as long as the first mode $\cos_2$ was the first to bifurcate. 

Our numerical continuation scheme has allowed us to discover another tuning curve encoded in the network that represents an orientation that is orthogonal to that of the LGN input. This neural illusion can be thought of as a ghost of the first pitchfork bifurcation that occurs when the sigmoid is centered (taking the value 0 at the origin) and opens up when the bias on the sigmoid is removed.

We have  shown that it was possible to drive a hypercolumn to the illusory state by adding some dynamics to the LGN input: this gave rise to two dynamical illusions, one relying on rotating the stimulus, the other relying on changing its contrast. This is a strong prediction of the model that could possibly be tested experimentally even though this seems difficult given the fact that the Ring Model does not take into account the lateral spatial connectivity that is present in the visual cortex and allows different hypercolumns of orientation to interact with each other, but see the above discussion.

It would be interesting to see if and how the illusions are modified when adding lateral spatial connections in a spatially organized network of Ring Models.

Finally our approach leads to a near complete understanding of the role of each parameter in the Ring Model: the shape of the connectivity function through the weights $J_i$, the threshold $\theta$, and the nonlinear gain $\lambda$.

\section*{Supporting information}
\subsection*{Numerical computation of the invariant functions}
 Let us say a few words about the practical computation of the invariant functions $\tilde B_0,a,b,c,d$. As we are interested in the tuning curves, using the estimates of \cite{veltz-faugeras:09}, we obtain that the $L^2$-norm $\norm{V^f}_2$ of the tuning curve is upperbounded by $14$ for the connectivity function shown in figure \ref{fig:OSmu0}. The relation $\norm{V^f}^2_2=\pi\left(v_0^2+\frac{J_1}{2}\pi_1+\frac{J_2}{2}\pi_2 \right)$ yields the estimation
	:
\[\norm{V^f}_\infty\leq|v_0|+\sqrt{J_1}(|v_1^{(1)}|+|v_1^{(2)}|)+\sqrt{|J_2|}(|v_2^{(1)}|+|v_2^{(2)}|)\overset{\rm Cauchy-Schwarz}{\leq}\frac{3}{\sqrt{\pi}}\norm{V^f}_2\leq 24\]
 for the same values of the parameters.
 The next step is to approximate the sigmoid $S$ by a polynomial $P$ on some interval $[-\alpha,\alpha]$ where the value of $\alpha$ is chosen so that $\lambda\norm{V^f}_\infty\leq \alpha$. As we need to observe the first two pitchfork bifurcations, reached for the values $\lambda_1 < \lambda_2$ of $\lambda$, see remark \ref{rm:numerics}, we need at least $\lambda=\lambda_2$ 
and, being a little bit conservative, we end up computing the solutions for $\lambda\in[0,0.6]$. This in turn requires $\alpha \approx 14$.
Note that the more accurate the approximation of $S$, the higher the degree of $P$ with the consequence that some numerical instabilities may develop since this implies raising small numbers to high powers.


$P$ is then expressed in the basis of the Chebychev polynomials as $P=\sum_i\alpha_iT_i$. The reason for this is that the Chebychev polynomials having rational coefficients, we can use, for example, the Groebner basis package of the symbolic computation package Maple to express the invariants $B_0,\,B_1,\,B_2$ as functions of $(v_0,\vec\pi)$. For example, we have:

\[
\begin{array}{lcl}
 B_0 &=&\frac{\varepsilon_0}{\pi}\int\limits_{-\frac{\pi}{2}}^\frac{\pi}{2}S\left[ \lambda v_0+\lambda \Re\left( \sqrt{J_1} z_1e^{-2piy} +\sqrt{J_1} z_2e^{-4piy}\right) \right] dy\\
&\approx& \frac{\varepsilon_0}{\pi}\int\limits_{-\frac{\pi}{2}}^\frac{\pi}{2}P\left[ \lambda v_0+\lambda \Re\left( \sqrt{J_1} z_1e^{-2piy} +\sqrt{J_1} z_2e^{-4piy}\right) \right] dy \\
&=&\frac{\varepsilon_0}{\pi}\sum_i\alpha_i\int\limits_{-\frac{\pi}{2}}^\frac{\pi}{2}T_i\left[ \lambda v_0+\lambda \Re\left( \sqrt{J_1} z_1e^{-2piy} +\sqrt{J_1} z_2e^{-4piy}\right) \right] dy\\
&\overset{Maple}{=}&\frac{\varepsilon_0}{\pi}\sum_i\alpha_i\tilde T_i(v_0,\vec\pi)
\end{array}
\]
The computation of $\tilde T_i(v_0,\vec\pi)=\int\limits_{-\frac{\pi}{2}}^\frac{\pi}{2}T_i\left[ \lambda v_0+\lambda \Re\left( \sqrt{J_1} z_1e^{-2piy} +\sqrt{J_1} z_2e^{-4piy}\right) \right] dy$ is done automatically by the Groebner basis package but requires that the coefficients of the polynomial $T_i$ be rational, not real. This justifies the Chebychev approximation. For $\alpha=14$ and an approximation error of $0.01$ ($\norm{S-P}_{\infty,[-14,14]}<0.01$), it gives a polynomials $P$ of degree $19$.
One important advantage of this method is that it does not require the vector $(v_0,\vec\pi)$ to be on the Orbit Space to do the computations, i.e. we can compute $\tilde B_0,a,b,c,d$ even for values of $(v_0,\vec\pi)$ that make no sense, e.g. $\pi_1<0$,  and then project the result on the Orbit Space. 
Note that the method is coherent since the results shown in figures~\ref{fig:OSmu0} and \ref{fig:OSmu1} obtained by numerical continuation do satisfy $\lambda\norm{V^f}_\infty<14$, that is, are consistent with the numerical approximation.

\subsection*{Numerical computation of the solutions of the nonlinear equations in the case $N=2$}\label{rm:numerics}
 In order to solve the nonlinear equations \eqref{eq:tcn2} for the tuning curves, we apply the strategy of \cite{veltz-faugeras:09}. The idea is to use a homotopy to solve the problem: we introduce a new parameter $\mu$ which translates $S$, i.e. $S_{(\mu)}\stackrel{def}{=}S_0+\mu(-\theta+ S(0))$ where $\theta$ is the threshold. Thus $S_{(0)}=S_0$ and $S_{(1)}=S-\theta$. This way we change the nonlinearity in \eqref{eq:discrete} in order to find the TCs analytically (notice that this translation only affects the first equation of \eqref{eq:discrete}). Indeed, when the nonlinearity is the centered sigmoid $S_0$ we obtain the trivial solution $V^f=0$ and  we can also compute the values of the nonlinear gain $\lambda$ where the pitchfork bifurcations occur. We can then numerically continue this trivial solution with respect to the parameters $(\lambda,\,\mu)$ to find the solutions of the equations with the ``correct'' nonlinearity, namely $S-\theta$. We then simply take a slice of the output of the continuation program for $\mu=1$ and obtain the dependency of the solutions w.r.t. the nonlinear gain $\lambda$. This approach, though numerically intensive, is very convenient because it automatically gives the bifurcated branches. It also allows to compute some non-connected branches of solutions. This strategy relies on the library TRILINOS, see the acknowledgements below.
\subsection*{Equivalent condition for the threshold}\label{ap:seuil}
Remember that our goal is to find a region in the plane $(\theta,J_1)$ where there exists a pair $(\lambda>0,\,v_0^f)$ such that :
 \[
 (E):\left\{
    	\begin{array}{lcl}
    	v_0^f&=&\varepsilon_0S(\lambda v_0^f)-\theta \\
    	1&=&\lambda S'(\lambda v_0^f) \frac{J_1}{2}
	\end{array}
\right.
\]
We work out the case $\varepsilon_0=1$, the other one being very similar. As $S'=S(1-S)$, the second equation (E.2) becomes : $1=\frac{\lambda J_1}{2}S(1-S)\overset{{\rm using} (E.1)}{=}\frac{\lambda J_1}{2}U(-1-U)$ where $U\overset{\rm def}{=}v_0^f+\theta$. This quadratic equation in $U$ has real solutions if and only if $\lambda J_1\geq 8$, and they are given by 
\[
 U_\pm=\frac{-1\pm\sqrt{1-\frac{8}{\lambda J_1}}}{2}
\]
We still have to verify that (E.1) is satisfied for at least one of these solutions. This yields an equivalent condition to (E) but with 3 variables instead of 4. 

For example, for $U_+$, we obtain the equation in $(\lambda,\,J_1,\,\theta)$ : $U_+=-S(\lambda U_+-\lambda\theta)$. Using brute force computation for $\lambda\in[0,30],\,\theta\in[0,1],\,J_1\in[0,10]$, we check when this is possible thereby obtaining the graphs shown in figure~\ref{fig:condtc}.
	
\subsection*{The width of the tuning curves}\label{ap:width}
\begin{lemma}
If $f(a,b)\overset{\rm def}{=}-{\frac {a+\ln  \left( 1+2\,{{\rm e}^{-a-b}} \right) }{b}}$, then the width at half height of the tuning curve is equal to 
 $\cos^{-1}\left[f(\lambda (v_0^f-\theta),\lambda \sqrt{\pi_1^fJ_1})\right]$
\end{lemma}
\begin{proof}
 By definition, we look for the angle $\varphi$ such that :
 
 $S\left[\lambda \left(v_0^f+ \sqrt{\pi_1^fJ_1} -\theta\right)\right] =2S\left[\lambda \left(v_0^f+\sqrt{\pi_1^fJ_1}\cos_2(\varphi)-\theta\right)\right] $.
 Setting $a=\lambda (v_0^f-\theta)$ and $b=\lambda \sqrt{\pi_1^fJ_1}$, it follows that $\cos(2\varphi)=f(a,b)$. Hence the half-width is given by $\varphi=\frac{1}{2}\cos^{-1}f(a,b)$.
\end{proof}


\section*{Acknowledgments}
The numerical continuation experiments were done  using the library TRILINOS (see \cite{sala-heroux-etal:04} and the website \textit{http://trilinos.sandia.gov}) using multiparameters continuation. 
We thank M.E. Hendersen for his valuable help in using the library Multifario (part of Trilinos). Much of our work would have proved difficult - if not impossible - without his help.
The authors wish to thank P.Chossat for fruitfull discussions regarding equivariance problems.\\
This work was partially funded by the ERC advanced grant NerVi number 227747.

\bibliographystyle{elsarticle-harv}

\begin{thebibliography}{10}
\providecommand{\url}[1]{\texttt{#1}}
\providecommand{\urlprefix}{URL }
\expandafter\ifx\csname urlstyle\endcsname\relax
  \providecommand{\doi}[1]{doi:\discretionary{}{}{}#1}\else
  \providecommand{\doi}{doi:\discretionary{}{}{}\begingroup
  \urlstyle{rm}\Url}\fi
\providecommand{\bibAnnoteFile}[1]{%
  \IfFileExists{#1}{\begin{quotation}\noindent\textsc{Key:} #1\\
  \textsc{Annotation:}\ \input{#1}\end{quotation}}{}}
\providecommand{\bibAnnote}[2]{%
  \begin{quotation}\noindent\textsc{Key:} #1\\
  \textsc{Annotation:}\ #2\end{quotation}}
\providecommand{\eprint}[2][]{\url{#2}}

\bibitem{hubel-wiesel:62}
Hubel D, Wiesel T (1962) Receptive fields, binocular interaction and functional
  architecture in the cat visual cortex.
\newblock J Physiol 160: 106--154.
\input{hubel-wiesel:62}

\bibitem{somers-nelson-etal:95}
Somers D, Nelson S, Sur M (1995) An emergent model of orientation selectivity
  in cat visual cortical simple cells.
\newblock Journal of Neuroscience 15: 5448.
\input{somers-nelson-etal:95}

\bibitem{ben-yishai-bar-or-etal:95}
Ben-Yishai R, Bar-Or R, Sompolinsky H (1995) Theory of orientation tuning in
  visual cortex.
\newblock Proceedings of the National Academy of Sciences 92: 3844--3848.
\input{ben-yishai-bar-or-etal:95}

\bibitem{hansel-sompolinsky:97}
Hansel D, Sompolinsky H (1997) Modeling feature selectivity in local cortical
  circuits.
\newblock Methods of neuronal modeling : 499--567.
\input{hansel-sompolinsky:97}

\bibitem{shriki-hansel-etal:03}
Shriki O, Hansel D, Sompolinsky H (2003) Rate models for conductance-based
  cortical neuronal networks.
\newblock Neural Computation 15: 1809--1841.
\input{shriki-hansel-etal:03}

\bibitem{ermentrout:98}
Ermentrout B (1998) Neural networks as spatio-temporal pattern-forming systems.
\newblock Reports on Progress in Physics 61: 353--430.
\input{ermentrout:98}

\bibitem{dayan-abbott:01}
Dayan P, Abbott L (2001) Theoretical Neuroscience : Computational and
  Mathematical Modeling of Neural Systems.
\newblock MIT Press.
\input{dayan-abbott:01}


\bibitem{bressloff-bressloff-etal:00}
Bressloff P, Bressloff N, Cowan J (2000) Dynamical mechanism for sharp
  orientation tuning in an integrate-and-fire model of a cortical hypercolumn.
\newblock Neural computation 12: 2473--2511.
\input{bressloff-bressloff-etal:00}

\bibitem{bressloff-cowan-etal:01}
Bressloff P, Cowan J, Golubitsky M, Thomas P, Wiener M (2001) Geometric visual
  hallucinations, {E}uclidean symmetry and the functional architecture of
  striate cortex.
\newblock Phil Trans R Soc Lond B 306: 299--330.
\input{bressloff-cowan-etal:01}

\bibitem{douglas-koch-etal:95}
Douglas R, Koch C, Mahowald M, Martin K, Suarez H (1995) Recurrent excitation
  in neocortical circuits.
\newblock Science 269: 981.
\input{douglas-koch-etal:95}

\bibitem{wilson-cowan:72}
Wilson H, Cowan J (1972) Excitatory and inhibitory interactions in localized
  populations of model neurons.
\newblock Biophys J 12: 1--24.
\input{wilson-cowan:72}

\bibitem{haragus-iooss:09}
Haragus M, Iooss G (2009) Local bifurcations, center manifolds, and normal
  forms in infinite dimensional systems.
\newblock EDP Sci. Springer Verlag UTX series.
\newblock To appear.
\input{haragus-iooss:09}

\bibitem{veltz-faugeras:09}
Veltz R, Faugeras O (2009) Local/global analysis of the stationary solutions of
  some neural field equations.
\newblock Technical report, arXiv.
\newblock \urlprefix\url{http://arxiv.org/pdf/0910.2247v1}.
\input{veltz-faugeras:09}

\bibitem{chossat-lauterbach:00}
Chossat P, Lauterbach R (2000) {Methods in Equivariant Bifurcations and
  Dynamical Systems}.
\newblock World Scientific Publishing Company.
\input{chossat-lauterbach:00}

\bibitem{carandini-ringach:97}
Carandini M, Ringach D (1997) Predictions of a recurrent model of orientation
  selectivity.
\newblock Vision Research 37: 3061--3071.
\input{carandini-ringach:97}

\bibitem{gerstner-kistler:02b}
Gerstner W, Kistler W (2002) Spiking Neuron Models.
\newblock Cambridge University Press.
\input{gerstner-kistler:02b}

\bibitem{blumenfeld-bibitchkov-etal:06}
Blumenfeld B, Bibitchkov D, Tsodyks M (2006) Neural network model of the
  primary visual cortex: From functional architecture to lateral connectivity
  and back.
\newblock Journal of computational neuroscience 20: 219--241.
\input{blumenfeld-bibitchkov-etal:06}

\bibitem{shelley-mclaughlin:02}
Shelley M, McLaughlin D (2002) Coarse-grained reduction and analysis of a
  network model of cortical response: I. drifting grating stimuli.
\newblock Journal of Computational Neuroscience 12: 97--122.
\input{shelley-mclaughlin:02}

\bibitem{sala-heroux-etal:04}
Sala M, Heroux MA, Day DM (2004) Trilinos tutorial.
\newblock Technical Report SAND2004-2189, Sandia National Laboratories.
\input{sala-heroux-etal:04}

\end{thebibliography}

\section*{Figure Legends}

\section*{Tables}

\end{document}